\documentclass[12pt, reqno]{amsart}
\usepackage{amsmath, amsthm, amscd, amsfonts, amssymb}
\usepackage{bm}
\usepackage[all,cmtip]{xy}

\newtheorem{theorem}{Theorem}[section]
\newtheorem{lemma}[theorem]{Lemma}
\newtheorem{proposition}[theorem]{Proposition}
\newtheorem{corollary}[theorem]{Corollary}
\theoremstyle{definition}
\newtheorem{definition}[theorem]{Definition}
\newtheorem{example}[theorem]{Example}

\newtheorem{remark}[theorem]{Remark}
\numberwithin{equation}{section}

\newcommand{\vp}{\varphi}

\newcommand{\clb}{\mathcal{B}}

\newcommand{\clh}{\mathcal{H}}

\newcommand{\cll}{\mathcal{L}}

\newcommand{\clm}{\mathcal{M}}

\newcommand{\clo}{\mathcal{O}}

\newcommand{\clq}{\mathcal{Q}}

\newcommand{\cls}{\mathcal{S}}

\newcommand{\clw}{\mathcal{W}}

\newcommand{\clz}{\mathcal{Z}}
\newcommand{\mi}{\mathfrak{I}}

\newcommand{\bk}{k}
\newcommand{\bl}{l}
\newcommand{\D}{\mathbb{D}}

\newcommand{\T}{\mathbb{T}}
\newcommand{\Z}{\mathbb{Z}}
\newcommand{\bS}{\mathbb{S}}

\newcommand{\raro}{\rightarrow}

\begin{document}

\setcounter{page}{1}

\title[Commutant lifting, interpolation, and perturbations]{Commutant lifting, interpolation, and perturbations on the polydisc}

\author[Deepak]{Deepak K. D.}
\address{Indian Statistical Institute, Statistics and Mathematics Unit, 8th Mile, Mysore Road, Bangalore, 560059,
India}
\email{dpk.dkd@gmail.com }

\author[Sarkar]{Jaydeb Sarkar}
\address{Indian Statistical Institute, Statistics and Mathematics Unit, 8th Mile, Mysore Road, Bangalore, 560059,
India}
\email{jay@isibang.ac.in, jaydeb@gmail.com}



\begin{abstract}
The fundamental theorem on commutant lifting due to Sarason does not carry over to the setting of the polydisc. This paper presents two classifications of commutant lifting in several variables. The first classification links the lifting problem to the contractivity of certain linear functionals. The second one transforms it into nonnegative real numbers via a distance formula. We also solve the Nevanlinna-Pick interpolation problem for bounded analytic functions on the polydisc. Along the way, we solve a perturbation problem for bounded analytic functions. Commutant lifting and interpolation on the polydisc solve two well-known problems in Hilbert function space theory.
\end{abstract}


\maketitle

\tableofcontents

\section{Introduction}\label{sec: 1}

Sarason's commutant lifting theorem \cite{Sarason} is fundamental, with significant applications to virtually every aspect of Hilbert function space theory. One of them is the Nevanlinna-Pick interpolation theorem on the open unit disc $\D = \{z \in \mathbb{C}: |z| < 1\}$, which we will quickly review before moving on to the lifting theorem. Given $m$ distinct points $\clz = \{z_1, \ldots, z_m\} \subset \D$ (interpolation nodes) and $m$ scalars $\clw = \{w_1, \ldots, w_m\} \subset \D$ (target data), there exists an analytic function $\vp :\D \raro \mathbb{C}$ (interpolating function) such that
\[
\sup_{z \in \D} |\vp(z)| \leq 1
\]
and
\[
\vp(z_i) = w_i,
\]
for all $i=1, \ldots, m$, if and only if the $m \times m$ \textit{Pick matrix} $\mathfrak{P}_{\clz, \clw}$ is positive semi-definite, where
\[
\mathfrak{P}_{\clz, \clw}:= \Big(\frac{1 - w_i \bar{w}_j}{1 - z_i \bar{z}_j}\Big)_{i,j=1}^m.
\]
This was proved by G. Pick \cite{Pick} more than a century ago. R. Nevanlinna \cite{RN} independently solved the same problem at a very similar time. The methods of Pick and Nevanlinna are different, interesting on their own, and still relevant. For instance, Pick focused on interpolation on the upper half-plane, whereas the Schur algorithm (see I. Schur \cite{Schur 1, Schur 2}) served as the driving force behind Nevanlinna's strategy \cite{Gar, AN}.

After four decades of Pick's paper, D. Sarason \cite{Sarason} provided a robust Hilbert function space theoretical foundation for Nevanlinna and Pick's analytic and algebraic methods for the solution of the interpolation problem. Sarason's elegant result, known as the \textit{commutant lifting theorem}, represents the commutant of model operators in terms of nicer operators (say Toeplitz operators) without changing the norms. To be more specific, let us recall that the Hardy space $H^2(\T^n)$ on the polydisc $\D^n$ is the space of all analytic functions $f$ on $\D^n$ such that
\[
\|f\|_2:= \Big(\sup_{0<r<1} \int_{\T^n} |f(rz)|^2 d\mu(z)\Big)^{\frac{1}{2}} < \infty,
\]
where $d\mu$ denotes the normalized Lebesgue measure on $\T^n$, $z = (z_1, \ldots, z_n)$, and $rz =(rz_1, \ldots, rz_n)$. Note that $\T^n = \partial \D^n$ is the \v{S}ilov boundary of $\D^n$. We denote the von Neumann algebra of essentially bounded Lebesgue measurable functions on $\T^n$ by $L^\infty(\T^n)$. The analytic counterpart of $L^\infty(\T^n)$ is $H^\infty(\D^n)$, where
\[
H^\infty(\D^n) = \{\vp \in \clo(\D^n): \|\vp\|_\infty:= \sup_{z \in \D^n} |\vp(z)| < \infty\},
\]
the Banach algebra of bounded analytic functions on $\D^n$ equipped with the supremum norm $\|\cdot\|_\infty$. The space $H^\infty(\D^n)$ is known to be one of the most sophisticated commutative Banach algebras. Even at the level of a single variable, many questions in analytic function theory and harmonic analysis remain to be settled (cf. \cite{Bourgain, Bala} and the references therein). Clearly, $H^\infty(\D^n) \subseteq H^2(\T^n)$. The class of functions of interest is the closed unit ball of $H^\infty(\D^n)$:
\[
\cls(\D^n) = \{\vp \in H^\infty(\D^n): \|\vp\|_\infty \leq 1\}.
\]
The members of $\cls(\D^n)$ are known as \textit{Schur functions}. Given $\vp \in H^\infty(\D^n)$, the \textit{analytic Toeplitz operator} $T_\vp$ on $H^2(\T^n)$ is defined by
\[
T_\vp f = \vp f,
\]
for all $f \in H^2(\T^n)$. In particular, for $\vp = z_i$, we get $T_{z_i}$ the multiplication operator by coordinate function $z_i$ on $H^2(\T^n)$, $i=1, \ldots, n$. The following equality describes how the commutant of $\{T_{z_i}\}_{i=1}^n$ connects the Banach algebra $H^\infty(\D^n)$ to $\clb(H^2(\T^n))$:
\[
\{T_{z_1}, \ldots, T_{z_n}\}' = \{T_\vp: \vp \in H^\infty(\D^n)\}.
\]
Moreover, it follows that \cite{Deepak}
\[
\|T_\vp\| = \|\vp\|_\infty \qquad (\vp \in H^\infty(\D^n)).
\]

We now return to the classical case where $n=1$. Let $\clq$ be a $T_z^*$-invariant closed subspace of $H^2(\T)$, and let $X$ be a bounded linear operator on $\clq$ (in short, $X \in \clb(\clq)$). Sarason's commutant lifting theorem states the following: Suppose $X$ commutes with the \textit{model operator} $P_{\clq} T_z|_{\clq} \in \clb(\clq)$, that is
\[
X (P_{\clq} T_z|_{\clq}) = (P_{\clq} T_{z}|_{\clq}) X.
\]
Then there exists $\vp \in H^\infty(\D)$ such that
\[
X = P_{\clq} T_{\vp}|_{\clq},
\]
and
\[
\|X\| = \|\vp\|_\infty.
\]
Here (and in what follows) $P_{\clq}$ denotes the orthogonal projection from $H^2(\T)$ onto $\clq$. In other words, along with $\|X\| = \|T_\vp\|$, the following diagram commutes:
\[
\xymatrix{
H^2(\T) \ar@{->}[rr]^{\displaystyle T_{\vp}} \ar@{<-}[dd]_{\displaystyle i_{\clq}}
&& H^2(\T) \ar@{->}[dd]^{\displaystyle P_{\clq}}    \\ \\
\clq \ar@{->}[rr]_{\displaystyle X} && \clq
}
\]
where $i_\clq : \clq \hookrightarrow H^2(\T)$ denotes the inclusion map. The Nevanlinna-Pick interpolation theorem then easily follows from this applied to zero-based finite-dimensional $T^*_z$-invariant subspaces of $H^2(\T)$ (cf. Subsection \ref{subsect: weak interpolation}). The most important aspect of Sarason's lifting theorem, however, is the lifting of the commutant of model operators to the commutant of $T_z$ keeping the norms the same.

Clearly, Sarason’s commutant lifting theorem stands out as one of the most natural and fundamental results in the field. Its impact extends far beyond the classical interpolation problem, with wide-ranging applications to operator and function theory. Notable examples include the Carath\'{e}odory-Fej\'{e}r interpolation problem, the Nehari interpolation problem, the von Neumann inequality, isometric dilations, and the Ando dilation, among others. The list of applications also includes areas such as control theory and electrical engineering \cite{FFGK, Helton}. However, when one deals with several variables, analogous questions related to the commutant lifting theorem and its many single-variable applications introduce distinct challenges. There are not many well-established theories that address these problems (however, see \cite{AM Crelle, Ball et, Barik et al, Das-Sarkar, DMM, vinnikov}). In fact, it is known that Sarason's commutant lifting theorem does not hold true in general in several variables (see Section \ref{sec: examples of hom qm}). Understanding the obstacle of commutant lifting in several variables is thus one of the most important problems in Hilbert function space theory.

In this paper, we solve the commutant lifting problem on $H^2(\T^n)$, $n \geq 1$. That is, given a closed subspace $\clq \subseteq H^2(\T^n)$ that is invariant under $T_{z_i}^*$, $i=1, \ldots, n$, we classify contractions $X \in \clb(\clq)$ satisfying the condition that
\[
X (P_{\clq} T_{z_i}|_{\clq}) = (P_{\clq} T_{z_i}|_{\clq}) X \qquad (i=1, \ldots, n),
\]
so that the following diagram commutes
\[
\xymatrix{
H^2(\T^n) \ar@{.>}[rr]^{\displaystyle T_{\vp}} \ar@{<-}[dd]_{\displaystyle i_{\clq}}
&& H^2(\T^n) \ar@{->}[dd]^{\displaystyle P_{\clq}}    \\ \\
\clq \ar@{->}[rr]_{\displaystyle X} && \clq
}
\]
for some $\vp \in \cls(\D^n)$. Several attempts have been made to solve this problem, but they appear to be quite abstract and only applicable to a smaller class of operators (or functions). The most notable one is perhaps the work of Ball, Li, Timotin, and Trent \cite{Ball et} (also see Clark \cite{Clark}). The class of functions considered in \cite{Ball et} is the so-called Schur-Agler class functions. This class is significantly smaller than even the polydisc algebra when $n > 2$, and it is the same as the Schur class when $n = 2$. Even in the $n=2$ case, however, the existing results are abstract. In the context of interpolation for the $n=2$ case, we refer the reader to the seminal papers by Agler \cite{Agler 1, Agler} (also, see the discussion following Theorem \ref{Thm: prel recov Sara}).

Our approach and solution to the commutant lifting problem are both concrete and function-theoretic. As part of the application, we solve the interpolation problem for Schur functions on $\D^n$. Moreover, in the context of Schur functions on $\D^n$, we also solve a perturbation problem. Like our commutant lifting theorem, all results are concrete and quantify the complexity of the problem by nonnegative real numbers.

Now we provide a more thorough summary of this paper's key contribution. Unless otherwise specified, we will always assume that $n \geq 1$ is a natural number. Given a Hilbert space $\clh$, set
\[
\clb_1(\clh) = \{T \in \clb(\clh): \|T\| \leq 1\}.
\]
Given a nonempty subset $S\subseteq H^2(\T^n)$, we define the conjugate space $S^{conj}$ as
\[
S^{conj} = \{\bar{f}: f \in S\}.
\]
Let $\cls \subseteq H^2(\T^n)$ be a closed subspace. We say that $\cls$ is a \textit{shift invariant subspace} (or submodule) if
\[
z_i \cls \subseteq \cls,
\]
for all $i=1, \ldots, n$. We say that $\cls$ is a \textit{backward shift invariant subspace} (or quotient module) if $\cls^\perp$ is a shift invariant subspace, or equivalently,
\[
T_{z_i}^* \cls \subseteq \cls,
\]
for all $i=1, \ldots, n$. Given a backward shift invariant subspace $\clq \subseteq H^2(\T^n)$, we define the \textit{model operator} $S_{z_i}$, for each $i=1, \ldots, n$, by
\[
S_{z_i} = P_{\clq} T_{z_i}|_{\clq}.
\]
Now we define lifting on backward shift invariant subspaces (also see Definition \ref{def: lift}).

\begin{definition}
Let $\clq \subseteq H^2(\T^n)$ be a backward shift invariant subspace, $X \in \clb_1(\clq)$, and suppose $X S_{z_i} = S_{z_i} X$ for all $i=1, \ldots, n$. If there exists $\vp \in \cls(\D^n)$ such that
\[
X = P_\clq T_\vp|_{\clq},
\]
then $X$ is said to have a lift, or to be liftable.
\end{definition}

Before we get into the main contribution of this paper, we need to familiarise ourselves with a few additional concepts. First, we define the closed subspace of ``mixed functions'' of $L^2(\T^n)$ as
\[
\clm_n = L^2(\T^n) \ominus ({H^2(\T^n)}^{conj} + H^2(\T^n)).
\]
This space has a significant role to perform in the entire paper. It is crucial to observe that $\clm_n \cap H^2(\T^n) = \{0\}$, and
\[
\clm_1 = \{0\}.
\]
Let $\clq \subseteq H^2(\T^n)$ be a backward shift invariant subspace. Set
\begin{equation}\label{eqn: intro M Q}
\clm_\clq = \clq^{conj} \dotplus (\clm_n \dotplus H^2_0(\T^n)),
\end{equation}
where
\[
H^2_0(\T^n) = H^2(\T^n) \ominus \{1\},
\]
the closed subspace of $H^2(\T^n)$ of functions vanishing at the origin. Note that $\dotplus$ signifies the skew sum of Banach spaces. In what follows, we treat $\clm_\clq$ as a subspace of the classical Banach space $L^1(\T^n)$, and denote it by $(\clm_\clq, \|\cdot\|_1)$. In other words
\[
(\clm_\clq, \|\cdot\|_1) \subset (L^1(\T^n), \|\cdot\|_1).
\]
Let $X \in \clb(\clq)$, and let
\begin{equation}\label{eqn: intro psi}
\psi = X(P_{\clq} 1).
\end{equation}
It is important to observe that $P_\clq 1 \neq 0$; otherwise, $1 \in \clq^\perp$, which would imply (in view of the fact that $z_i \clq^\perp \subseteq \clq^\perp$ for all $i=1, \ldots, n$)
\[
\clq^\perp = H^2(\T^n),
\]
and consequently, $\clq = \{0\}$--a contradiction. Define a functional $X_\clq: (\clm_\clq, \|\cdot\|_1) \longrightarrow \mathbb{C}$ by
\[
X_\clq f = \int_{\T^n} \psi f \;d\mu \qquad (f \in \clm_\clq),
\]
where $d\mu$ denotes the normalized Lebesgue measure on $\T^n$. Finally, set
\[
\widetilde{\clm}_{\clq, X} = ({\clq}^{conj} \ominus \{\bar{\psi}\}) \dotplus (\clm_n \dotplus H^2_0(\T^n)),
\]
and again treat it as a subspace of $L^1(\T^n)$:
\[
(\widetilde{\clm}_{\clq, X}, \|\cdot\|_1) \subset (L^1(\T^n), \|\cdot\|_1).
\]
Now that we have these notations, we can say how the lifting of commutants in higher dimensions is classified (see Theorem \ref{thm: combine CLT}):

\begin{theorem}\label{intro: thm 1}
Let $\clq \subseteq H^2(\T^n)$ be a backward shift invariant subspace and let $X \in \clb(\clq)$ be a contraction. Suppose $X S_{z_i} = S_{z_i} X$ for all $i=1, \ldots, n$. The following conditions are equivalent:
\begin{enumerate}
\item $X$ admits a lift.
\item $X_\clq: (\clm_\clq, \|\cdot\|_1) \longrightarrow \mathbb{C}$ is a contractive functional, where
\[
X_\clq f = \int_{\T^n} \psi f \;d\mu \qquad (f \in \clm_\clq).
\]
\item $\text{dist}_{L^1(\T^n)}\Big(\frac{\bar \psi}{\|\psi \|_2^2}, \widetilde{\clm}_{\clq, X}\Big) \geq 1$.
\end{enumerate}
\end{theorem}

This solves the long-standing commutant lifting problem for $H^2(\mathbb{T}^n)$, $n >1$. We believe that the technique used to prove our lifting theorem is interesting on its own. In fact, within the framework of the application to a different flavor, we solve a perturbation problem of independent interest: Given a nonzero function $f \in H^2(\T^n)$, does there exist $g \in H^2(\T^n)$ such that
\[
f + g \in \cls(\D^n)?
\]
Of course, to avoid triviality (that $g = - f$, for instance), we must assume that $g \in \{f\}^\perp$. Set
\[
\cll_n = \clm_n \oplus H^2_0(\T^n),
\]
and treat it as a subspace of $L^1(\T^n)$. Theorem \ref{thm: perturbation} presents a complete solution to this problem:

\begin{theorem}\label{intro: thm 2}
Let $f \in H^2(\T^n)$ be a nonzero function. Then there exists $g \in \{f\}^\perp$ such that
\[
f + g \in \cls(\D^n),
\]
if and only if
\[
\text{dist}_{L^1(\T^n)}\Big(\frac{\bar f}{\|f\|_2^2}, \cll_n \Big) \geq 1.
\]
\end{theorem}

Now we will explain the solution to the interpolation problem, which also resolves the long-standing question on interpolation with Schur functions as interpolating functions on $\D^n$, $n > 1$. We will start by laying the groundwork. Recall that $H^2(\T^n)$ is a reproducing kernel Hilbert space corresponding to the \textit{Szeg\"{o} kernel} $\mathbb{S}: \D^n \times \D^n \raro \mathbb{C}$, where
\[
\mathbb{S}(z,w) = \prod_{i=1}^n \frac{1}{1 - z_i \bar{w}_i} \qquad (z, w \in \D^n).
\]
For each $w \in \D^n$, define $\mathbb{S}(\cdot,w) : \D^n \raro \mathbb{C}$ by $(\mathbb{S}(\cdot,w))(z) = \mathbb{S}(z,w)$ for all $z \in \D^n$. In view of the standard reproducing kernel property, it follows that $\{\mathbb{S}(\cdot,w): w \in \D^n\} \subseteq H^2(\T^n)$ is a set of linearly independent functions, and
\[
\mathbb{S}(z,w) = \langle \mathbb{S}(\cdot,w), \mathbb{S}(\cdot,z) \rangle_{H^2(\T^n)},
\]
for all $z, w \in \D^n$. Given a set of distinct points $\clz = \{z_1, \ldots, z_m\} \subset \D^n$, we define an $m$-dimensional subspace of $H^2(\T^n)$ as
\[
\clq_\clz = \text{span} \{\mathbb{S}(\cdot, z_j): j =1, \ldots, m\}.
\]
It follows that $\clq_\clz$ is a backward shift invariant subspace of $H^2(\T^n)$. Define
\[
\clm_{\clq_\clz} = {\clq}^{conj}_{\clz} \dotplus (\clm_n \dotplus H^2_0(\T^n)).
\]
In addition, given a set of scalars $\{w_i\}_{i=1}^m \subset \D$, define $X_{\clz, \clw} \in \clb(\clq_\clz)$ by
\[
X_{\clz, \clw}^* \mathbb{S}(\cdot, z_j) = \bar{w}_j \mathbb{S}(\cdot, z_j) \qquad (j=1, \ldots,m).
\]
The fact that $X_{\clz, \clw}$ on $\clq_\clz$ is a natural operator and that it meets the crucial condition that $X_{\clz, \clw} S_{z_i} = S_{z_i} X_{\clz, \clw}$, $i=1, \ldots, m$, is noteworthy (see Lemma \ref{eqn: Y star S}).

Here is a summary of our main interpolation results (see Theorem \ref{thm: interpolation} and Theorem \ref{thm: quantitative inter}):

\begin{theorem}\label{Thm: prel inter 2}
Let $\clz = \{z_i\}_{i=1}^m \subset \D^n$ be $m$ distinct points, and let $\clw = \{w_i\}_{i=1}^m \subset \D$ be $m$ scalars. The following conditions are equivalent:

\begin{enumerate}
\item There exists $\vp \in \cls(\D^n)$ such that $\vp(z_i) = w_i$ for all $i=1, \ldots, m$.
\item  $\mi_{\clz, \clw}: (\clm_{\clq_\clz}, \|\cdot\|_1) \raro \mathbb{C}$ is a contraction, where
\[
\mi_{\clz, \clw} f = \int_{\T^n} \psi_{\clz,\clw} f \,d \mu,
\]
for all $f \in \clm_{\clq_\clz}$, and
\[
\psi_{\clz,\clw} = \sum_{i=1}^{m} c_i \bS(\cdot, z_i),
\]
and the scalar coefficients $\{c_i\}_{i=1}^m$ are given by
\[
\begin{bmatrix}
c_1
\\
c_2\\
\vdots
\\
c_m
\end{bmatrix}
=
\begin{bmatrix}
\bS(z_1, z_1) & \bS(z_1, z_2) & \cdots & \bS(z_1, z_m)
\\
\bS(z_2, z_1) & \bS(z_2, z_2) & \cdots & \bS(z_2, z_m)
\\
\vdots & \ddots & \ddots & \vdots
\\
\bS(z_m, z_1) & \bS(z_m, z_2) & \cdots & \bS(z_m, z_m)
\end{bmatrix}^{-1}
 \begin{bmatrix}
w_1
\\
w_2\\
\vdots
\\
w_m
\end{bmatrix}.
\]
\item Let $\psi := X_{\clz, \clw} (P_{\clq_\clz} 1)$, and let
\[
\widetilde{\clm}_{\clz, \clw} := ({\clq}^{conj}_{\clz} \ominus \{\bar{\psi}\}) \dotplus (\clm_n \dotplus H^2_0(\T^n)).
\]
Then
\[
\text{dist}_{L^1(\T^n)}\Big(\frac{\bar \psi}{\|\psi\|_2^2}, \widetilde{\clm}_{\clq_\clz}\Big) \geq 1.
\]
\end{enumerate}
\end{theorem}

Note that the matrix in part (2) of the above theorem is the inverse of the Gram matrix
\[
\Big(\bS(z_i, z_j)\Big)_{i,j=1}^m,
\]
corresponding to the $m$ Szeg\"{o} kernel functions $\{\bS(\cdot, z_i)\}_{i=1}^m$. Also, observe that part (3) provides a useful quantitative criterion to check interpolation on the polydisc. Indeed, as shown in Theorem \ref{thm: eg interpolation}, the quantitative criterion yields examples of interpolation on $\D^n$, $n \geq 2$. Notable is the fact that interpolating functions in this particular case are polynomials.

It is noteworthy that the answer to natural questions, as in Theorems \ref{intro: thm 1}, \ref{intro: thm 2}, and \ref{Thm: prel inter 2}, has a connection with a distance formula. This is a common and classical occurrence. The classical Nehari theorem \cite{N}, for example, establishes a direct link with such a distance function. Another instance is the celebrated Adamyan-Arov-Krein formulae \cite{AAK, AAK 2, AAK1}. We will comment some more at the end of Section \ref{sec: perturbation}.

In Theorem \ref{thm: recov sarason}, we recover Sarason's lifting theorem as an application to Theorem \ref{intro: thm 1}, resulting in yet another proof of the classical lifting theorem:

\begin{theorem}\label{Thm: prel recov Sara}
Let $\clq \subseteq H^2(\T)$ be a backward shift invariant subspace, and let $X \in \clb_1(\clq)$. If
\[
X S_{z} = S_{z} X,
\]
then $X$ is liftable.
\end{theorem}

In the proof of the above theorem, $\clm_\clq$ (defined as in \eqref{eqn: intro M Q}) admits a more compact form, namely
\[
\clm_\clq = \overline{\vp}(z H^2(\T)),
\]
where $\vp \in H^\infty(\D)$ is an inner function (that is, $|\vp| =1$ on $\T$ a.e.) and $\clq = (\vp H^2(\T))^\perp$. Moreover, we employ all the standard one variable types of machinery like the Beurling theorem, inner-outer factorizations \cite{Beurling, Gar}, etc. On the one hand, this is to be expected given that Sarason uses similar tools for his lifting theorem. This, on the other hand, explains both the challenges associated with the commutant lifting theorem and the potential for extensions of relevant function theoretic results on the polydisc.

Here are some additional facts and thoughts regarding the commutant lifting and interpolation problems, as well as the context of our approach. In 1968, Sz.-Nagy and Foia\c{s} \cite{NF 68} generalized the Sarason lifting theorem to vector-valued Hardy spaces. In subsequent papers, many researchers presented a variety of alternative proofs of independent interest (cf. \cite{Arocena, DMP, Sarason 1}). However, the dilation theory (pioneered by Halmos \cite{Halmos} and advanced by Sz.-Nagy \cite{Nagy}) is the primary technique employed in all of these papers which is powerful enough to negate the heavy use of function theoretic tools. For different versions of the commutant lifting theorem and its applications, we refer to Bercovici, Foia\c{s} and Tannenbaum \cite{BFT}, and the monographs by Nikolski \cite{Nik book 1}, Sz.-Nagy and Foia\c{s} \cite{NF Book}, and Foia\c{s} and Frazho \cite{FF} (also see Nikolski and Volberg \cite{NV} and Seip \cite{KS}). We refer to \cite{JKM} for an interpolation theorem in the context of a family of positive semi-definite kernels defined on a set.

In several variables, the earlier approach to the lifting theorem also appears to be dilation theoretic or under the assumption of von Neumann inequality, where dilation theory and von Neumann inequality for commuting contractions are complex subjects in and of themselves.

On the other hand, if the solution to the interpolation problem on $\D^n$, $n \geq 1$, is sought in terms of the Pick matrix's positive semi-definiteness (see Subsection \ref{subsect: weak interpolation} for the notion of Pick matrices on $\D^n$), then the interpolation problem becomes equivalent to the commutant lifting theorem on finite-dimensional zero-based subspaces (cf. Proposition \ref{prop: Pick set and CLT}). Consequently, in one variable, thanks to Sarason, the commutant lifting property, the Pick positivity, and the solution to the interpolation problem appear to be inextricably linked. In several variables, since the commutant lifting property tends to be inconsistent (cf. Section \ref{sec: examples of hom qm}), it is necessary to decouple the positivity of the Pick matrix from the interpolation problem. In some ways, these observations seek a different perspective on the several variables interpolation problem, one that is not as similar to the classical case of positivity of the Pick matrix (nor even positivity of a family of Pick matrices as in \cite{Abr, CLW, DMM, Hamliton, JKM}). As a consequence, we approach the problem from a completely different angle: more along the function theoretic path pioneered by Sarason. The difficulty here, of course, is dealing with the greater intricacy of several complex variables as well as the lack of all standard one variable tools.

Finally, a few words about this paper's methodology. We heavily use the duality of classical Banach spaces, namely
\[
(L^1(\T^n))^* \cong L^\infty(\T^n).
\]
Other common tools used in this paper include the classical Hahn-Banach theorem, the geometry of Banach spaces, and the Hilbert function space theory.

The remainder of the paper is structured as follows. Section \ref{sec: prel} introduces some preliminary concepts. Section \ref{sec: examples of hom qm} outlines explicit examples of non-liftable maps. Section \ref{sec: CLT Classification} presents classifications of commutant lifting. Section \ref{sec: perturbation} solves the perturbation problem of $H^2(\T^n)$-functions in terms of Schur functions. Section \ref{sec: weak lift} presents the first classification of the interpolation on $\D^n$. A quantitative classification for interpolation is presented in Section \ref{sec: quant inter}. In the same section, by using the quantitative classification, we provide examples of interpolation on $\D^n$, $n\geq 2$. The commutant lifting theorem on $\D^n$ is tested in Section \ref{subsect: example verification} with some concrete examples. As an application to our main commutant lifting theorem, Section \ref{sec: recover Sarason} provides new proof for the classical lifting theorem. In Section \ref{sec: con remarks} we make some general observations such as the Carath\'{e}odory-Fej\'{e}r interpolation problem, weak interpolation, and decomposing a polynomial as a sum of bounded analytic functions. Section \ref{sec: final} concludes with some closing remarks and thoughts on some other known results.

The paper contains an abundance of examples and counterexamples, as well as numerous auxiliary results of independent interest in both one and several variables. This paper is nearly self-contained.

\section{Preliminaries}\label{sec: prel}

In this section, we will introduce some necessary Hilbert function space theoretic preliminaries. These include Hardy space, submodules, quotient modules, and a formal definition of lifting. We begin by looking at the Hardy space. We again remind the reader that throughout the paper, $n$ will denote a natural number, and (unless otherwise stated) we always assume that $n \geq 1$.

We denote as usual by $L^2(\T^n)$ the space of square-integrable functions on $\T^n$. Recall that $\T^n$ is the \v{S}ilov boundary of $\D^n$. The \textit{Hardy space} $H^2(\T^n)$ is the closed subspace of $L^2(\T^n)$ consisting of those functions whose Fourier coefficients vanish off $\Z_+^n$. More specifically, consider $f \in L^2(\T^n)$ with Fourier series representation
\[
f = \sum_{k \in \Z^n} a_{k} z^{k} \qquad (z \in \T^n),
\]
where $z^{k} = z_1^{k_1} \cdots z_n^{k_n}$ for all $k = (k_1, \ldots, k_n) \in \Z^n$. Then $f \in H^2(\T^n)$ if and only if $a_{k} = 0$ whenever at least one of the $k_j$, $j = 1, \ldots, n$, in $k = (k_1, \ldots, k_n)$ is negative. The usage of radial limits is another neat way to represent the Hardy space (see Rudin \cite{WR}). In other words, we will identify $H^2(\T^n)$ with $H^2(\D^n)$, the Hilbert space of analytic functions $f \in \clo(\D^n)$ such that
\begin{equation}\label{eqn: norm f in 2}
\|f\|_2 := \Big(\sup_{0<r<1} \int_{\T^n} |f(rz)|^2 d\mu(z)\Big)^{\frac{1}{2}} < \infty,
\end{equation}
where $d\mu$ denotes the normalized Lebesgue measure on $\T^n$, and $rz = (rz_1, \ldots, r z_n)$. The identification is canonical, that is, given $f \in H^2(\D^n)$, the radial limit
\[
\tilde{f}(z) := \lim_{r \raro 1^{-}} f(rz),
\]
exists for almost every $z \in \T^n$, and $\tilde{f} \in H^2(\T^n)$, and vice-versa. In what follows (and unless otherwise stated) we will not distinguish between $f \in \clo(\D^n)$ satisfying \eqref{eqn: norm f in 2} and its radial limit representation $\tilde{f} \in H^2(\T^n)$. Therefore, we will not distinguish between $H^2(\T^n)$ and $H^2(\D^n)$ and will use the same notation $H^2(\T^n)$ for both.

It is frequently useful to represent $H^2(\T^n)$ as the Hilbert space of functions given by a power series in $z_1, \ldots, z_n$ with square-summable coefficients, that is
\[
H^2(\T^n) = \left\{\sum_{k \in \Z_+^n} a_{k} z^{k} \in \clo(\D^n): \sum_{k \in \Z_+^n} |a_{k}|^2 < \infty\right\}.
\]
The Hardy space $H^2(\T^n)$ is equipped with the tuple of multiplication operators by coordinate functions $\{z_1, \ldots, z_n\}$, which we denote by $(T_{z_1}, \ldots, T_{z_n})$. Therefore, by definition, we have
\[
(T_{z_i}f)(w) = w_i f(w),
\]
for all $f \in H^2(\T^n), w \in \D^n$, and $i=1, \ldots, n$. It is easy to see that $(T_{z_1}, \ldots, T_{z_n})$ is an $n$-tuple of commuting isometries, that is
\[
T_{z_i}^* T_{z_i} = I_{H^2(\T^n)}, \text{ and } T_{z_i} T_{z_j} = T_{z_j} T_{z_i},
\]
for all $i,j = 1, \ldots, n$. We will also need to use the doubly commutative property
\[
T_{z_i}^* T_{z_j} = T_{z_j} T_{z_i}^* \qquad (i \neq j).
\]
From the analytic function space perspective, recall that $H^2(\T^n)$ is a reproducing kernel Hilbert space corresponding to the \textit{Szeg\"{o} kernel} $\bS$ on $\D^n$, where
\[
\mathbb{S}(z,w) = \prod_{i=1}^n \frac{1}{1 - z_i \bar{w}_i} \qquad (z, w \in \D^n).
\]
For each $w \in \D^n$, the \textit{kernel function} $\mathbb{S}(\cdot, w) : \D^n \raro \mathbb{C}$ defined by
\[
(\mathbb{S}(\cdot, w))(z) = \mathbb{S}(z, w) \qquad (z \in \D^n),
\]
generates the joint eigenspace of the backward shifts, that is
\begin{equation}\label{eqn: eigen vec kernel}
\bigcap_{i=1}^n \ker (T_{z_i} - w_i I_{H^2(\T^n)})^* = \mathbb{C}\mathbb{S}(\cdot, w).
\end{equation}
The above equality essentially follows from the fact that
\begin{equation}\label{eqn: eigen vec kernel 2}
T_{z_i}^* \mathbb{S}(\cdot, w) = \bar{w}_i \mathbb{S}(\cdot, w),
\end{equation}
for all $w \in \D^n$ and $i=1, \ldots, n$, and
\[
\sum_{k \in \{0,1\}^n} (-1)^{|k|} T_z^{k} T_z^{*k} = P_{\mathbb{C}},
\]
where $P_{\mathbb{C}}$ is the orthogonal projection onto the space of constant functions, and $T_z^k = T_{z_1}^{k_1}\cdots T_{z_n}^{k_n}$ for all $k \in \{0,1\}^n \subset \Z_+^n$. Moreover, the set of kernel functions $\{\mathbb{S}(\cdot, w): w \in \D^n\}$ forms a total set in $H^2(\T^n)$ and satisfies the \textit{reproducing property}
\begin{equation}\label{eqn: rep prop}
f(w) = \Big\langle f, \mathbb{S}(\cdot, w) \Big\rangle_{H^2(\T^n)},
\end{equation}
for all $f \in H^2(\T^n)$ and $w \in \D^n$.

Recall from Section \ref{sec: 1} that a closed subspace $\clq \subseteq H^2(\T^n)$ is called a \textit{quotient module} if $T_{z_i}^* \clq \subseteq \clq$ for all $i=1, \ldots, n$. A closed subspace $\cls \subseteq H^2(\T^n)$ is called a \textit{submodule} if $z_i \cls \subseteq \cls$ for all $i=1, \ldots, n$. Equivalently, $\cls^\perp \cong H^2(\T^n)/\cls$ is a quotient module. In summary, we have the following identifications:
\[
\{\text{submodules}\} \longleftrightarrow \{\text{shift invariant subspaces}\},
\]
and
\[
\{\text{quotient modules}\} \longleftrightarrow \{\text{backward shift invariant subspaces}\}.
\]
The classical \textit{Laurent operator} $L_\vp$ with symbol $\vp \in L^\infty(\T^n)$ is the bounded linear operator on $L^2(\T^n)$ defined by
\[
L_\vp f = \vp f,
\]
for all $f \in L^2(\T^n)$. The corresponding \textit{Toeplitz operator} is the compression of $L_\vp$ to $H^2(\T^n)$, that is
\[
T_\vp f = P_{H^2(\T^n)} (\vp f),
\]
for all $f \in H^2(\T^n)$. As usual, $P_{H^2(\T^n)}$ denotes the orthogonal projection from $L^2(\T^n)$ onto $H^2(\T^n)$. Recall that (see \cite{Deepak})
\begin{equation}\label{eqn: Toeplitz norm equality}
\|T_\vp\|_{\clb(H^2(\T^n))} = \|L_\vp\|_{\clb(L^2(\T^n))} = \|\vp\|_\infty,
\end{equation}
for all $\vp \in L^\infty(\T^n)$. It is useful to point out that the Toeplitz operator with analytic symbol $\vp \in H^\infty(\D^n)$ is given by
\[
T_\vp = L_\vp|_{H^2(\T^n)}.
\]
This follows from the general fact that if $\cls$ is a submodule of $H^2(\T^n)$, then $\vp \cls \subseteq \cls$ for all $\vp \in H^\infty(\D^n)$. Finally, given a quotient module $\clq$ of $H^2(\T^n)$ and an analytic symbol $\varphi \in H^\infty(\D^n)$, we define the compression operator $S_\varphi$ on $\clq$ by
\[
S_\vp = P_{\clq} T_\vp|_{\clq}.
\]
In particular, for each $i=1, \ldots, n$, we have the compression of $T_{z_i}$ on $\clq$ as
\[
S_{z_i} = P_{\clq} T_{z_i}|_{\clq}.
\]
Clearly, $S_\vp S_{z_i} = S_{z_i} S_\vp$ for all $i=1, \ldots, n$. From this point of view, we also call that $S_\vp$ a module map. In general:

\begin{definition}\label{def: module map}
Let $\clq$ be a quotient module of $H^2(\T^n)$. An operator $X \in \clb(\clq)$ is said to be a module map if
\[
X S_{z_i} = S_{z_i} X \qquad (i=1, \ldots, n).
\]
\end{definition}

Another common name for module maps is \textit{truncated Toeplitz operators} (even with $L^\infty(\T^n)$-symbols). See Sarason \cite{Sarason 2} and also the classic by Brown and Halmos \cite{BH}.

We conclude this section with the definition of the central concept of this paper. Given a Hilbert space $\clh$, recall again that
\[
\clb_1(\clh) = \{T \in \clb(\clh): \|T\| \leq 1\}.
\]

\begin{definition}\label{def: lift}
Let $\clq \subseteq H^2(\T^n)$ be a quotient module and let $X \in \clb_1(\clq)$ be a module map. If there is a $\vp \in \cls(\D^n)$ such that
\[
X = S_\vp,
\]
then we say that $X$ has a lift, or $X$ is liftable, or $X$ admits a lift. We also say that $\vp$ is a lift of $X$.
\end{definition}

In the case of $n=1$, Sarason's result states that contractive module maps are always liftable. In the following section, we demonstrate that such a statement is no longer true whenever $n > 1$.

\section{Homogeneous quotient modules}\label{sec: examples of hom qm}

The purpose of this section is to outline explicit and basic examples of non-liftable module maps on quotient modules of $H^2(\T^n)$, $n > 1$. Our quotient modules are as simple as homogenous quotient modules and the module maps are compressions of homogeneous polynomials. We begin with a (probably known) classification of inner polynomials on $\D^n$. A function $\vp \in H^\infty(\D^n)$ is called \textit{inner} if $\vp$ is unimodular a.e. on $\T^n$ (in the sense of radial limits).

\begin{lemma}\label{lem: inner poly}
Let $p$ be a nonzero polynomial in $\mathbb{C}[z_1, \ldots, z_n]$. Then $p$ is inner if and only if
\[
p = \text{unimodular constant} \times \text{monomial}.
\]
\end{lemma}
\begin{proof}
We assume $n>1$ because the $n=1$ case is simpler and follows the same line of proof as the $n>1$ case. By definition, $p$ is inner if and only if $|p| = 1$ on $\T^n$. The sufficient part is now trivial. For the reverse direction, assume that $p$ is inner. If $p$ is a constant multiple of a monomial, then passing to the boundary value, the assertion will follow immediately. Therefore, assume that $p$ has more than one term. There exists $N_1 \in \mathbb{N}$ such that
\[
p = \sum_{j=0}^{N_1} z_1^j p_j,
\]
where $p_j \in \mathbb{C}[z_2, \ldots, z_n]$ for all $j = 0, 1, \ldots, N_1$, and
\[
p_{N_1} \neq 0.
\]
Here we are assuming without loss of generality that $p$ has a monomial term with $z_1$ as a factor (otherwise, we pass to the same but with respect to $z_2$ and so on). Since $p$ is inner, on $\T^n$, we have
\[
\begin{split}
1 & = p \bar{p}
\\
& = \bar{z}_1^{N_1} (p_0 \overline{p_{N_1}} + (\overline{p_{N_1-1}} p_0 +  p_1 \overline{p_{N_1}}) z_1 + \cdots + \overline{p_0}  p_{N_1} z_1^{2N_1}).
\end{split}
\]
This implies
\[
p_0 \overline{p_{N_1}} = 0,
\]
and hence $p_0 = 0$. Continuing exactly in the same way, we obtain that
\[
p = z_1^{N_1} p_{N_1},
\]
for some $p_{N_1} \in \mathbb{C}[z_2, \ldots, z_n]$. Applying the above recipe to $p_{N_1}$, we get $p_{N_1} = z_2^{N_2} p_{N_2}$ for some $N_2 \in \Z_+$ and $p_{N_2} \in \mathbb{C}[z_3, \ldots,, z_n]$. Hence
\[
p = z_1^{N_1} z_2^{N_2} p_{N_2}.
\]
Therefore, applying this method repeatedly, we finally deduce that $p$ is a unimodular constant multiple of  some monomial.
\end{proof}

Now we turn to the construction of the quotient modules of interest. As is well known and also evident from the definition of the Hardy space, polynomials are dense in $H^2(\T^n)$, that is
\[
H^2(\T^n) = \overline{\mathbb{C}[z_1, \ldots, z_n]}^{L^2(\T^n)}.
\]
Therefore, the standard grading on $\mathbb{C}[z_1, \ldots, z_n]$ induces a graded structure on $H^2(\T^n)$. We are essentially going to exploit this simple property in our construction of module maps. For each $t \in \Z_+$, denote by $H_t \subseteq \mathbb{C}[z_1, \ldots, z_n]$ the complex vector space of homogeneous polynomials of degree $t$. We have the vector space direct sum
\[
\mathbb{C}[z_1, \ldots, z_n] = \bigoplus_{t \in \Z_+} H_t.
\]
We consider from now on the finite-dimensional subspace $\clh_t$ as a closed subspace of $H^2(\T^n)$. Also, for each $m \in \mathbb{N}$, set
\[
\clq_m = \bigoplus_{t=0}^m H_t.
\]
Since $T_{z_i}^* \clq_m \subseteq \clq_m$ for all $m \geq 1$, it follows that $\clq_m \subseteq \mathbb{C}[z_1, \ldots, z_n]$ is a finite-dimensional quotient module of $H^2(\T^n)$, and $\text{deg} f \leq m$ for all $f \in \clq_m$. Fix $m \in \mathbb{N}$ and fix a homogeneous polynomial of degree $m$ as
\[
p = \sum_{|\bk| = m} a_{\bk} z^{\bk} \in H_m.
\]
Suppose that $\|p\|_2 = 1$. By the definition of the norm on $H^2(\T^n)$, we have
\[
\sum_{|\bk| = m} |a_{\bk}|^2 = 1.
\]\
We aim at investigating the lifting of the module map
\[
S_p = P_{\clq_m} T_p|_{\clq_m}.
\]
By $S_p f = P_{\clq_m} (pf)$, $f \in \clq_m$, we have on one hand $S_p 1 = p$, and on the other hand
\[
S_p f = 0,
\]
for all $f \in \clq_m$ such that $f(0) = 0$. Therefore, $\ker S_p = \clq_m \ominus \mathbb{C}$ or, equivalently
\[
\ker S_p = \bigoplus_{t=1}^m H_t.
\]
This allows us to conclude that
\begin{equation}\label{eqn: norm of Sp 1}
\|S_p\| = 1.
\end{equation}
We recall in passing that $\|T_\vp\|_{\clb(H^2(\T^n))} = \|\vp\|_\infty$ for all $\vp \in L^\infty(\T^n)$  (see \eqref{eqn: Toeplitz norm equality}).

\begin{theorem}
$S_p$ admits a lift if and only if $p$ is a unimodular constant multiple of a monomial.
\end{theorem}
\begin{proof}
Suppose $S_p$ is liftable. There exists $\vp \in \cls(\D^n)$ such that $S_p = S_\vp$. Then
\[
S_p = S_\vp = P_{\clq_m} T_\vp|_{\clq_m},
\]
and
\[
\|\vp\|_\infty \leq 1.
\]
Note that $1 \in \clq_m$. Since $S_p 1 = p$, it is clear that $P_{\clq_m} \vp = p$, and hence there exists $\psi \in \clq_m^\perp$ such that
\[
\vp = p \oplus \psi \in \clq_m \oplus \clq_m^\perp.
\]
It is a well known general fact that $\|\vp\|_2 \leq \|\vp\|_\infty$. Indeed
\[
\begin{split}
\|\vp\|_2 & = \|T_\vp 1\|_2
\\
& \leq \|T_\vp\|_{\clb(H^2(\T^n))} \|1\|_2
\\
& = \|\vp\|_\infty.
\end{split}
\]
Now that $\|p\|_2 = 1$, we compute
\[
\begin{split}
1 + \|\psi\|_2^2 & = \|\vp\|^2_2
\\
& \leq \|\vp\|^2_\infty
\\
& \leq 1,
\end{split}
\]
which implies $\psi = 0$. Therefore
\[
\vp = p \in \clq_m.
\]
By using the same computation (or the standard norm equality) as above, we have
\[
1 = \|p\|_2  \leq \|p\|_\infty = \|\vp\|_\infty  = 1,
\]
which implies that $\|p\|_\infty = 1$. This combined with
\[
\|T_p 1\|_2 = \|p\|_2 = 1,
\]
imply that the Toeplitz operator $T_p$ is norm attaining. Consequently \cite[Corollary 2.3]{Deepak}, $p$ is inner (as $p \in H^\infty(\D^n)$). Then by Lemma \ref{lem: inner poly} we conclude  that $p$ is a unimodular constant multiple of a monomial. The converse is obvious.
\end{proof}

The following corollary is now straightforward. Here we need to assume that $n>1$.

\begin{corollary}\label{cor: hom poly fail}
Let $n > 1$. Let $p$ be a homogeneous polynomial of degree $m$ and assume that $\|p\|_2 = 1$. Suppose
\[
p = \sum_{|\bk| = m} a_{\bk} z^{\bk} \in H_m.
\]
If $a_{\bk}, a_{\bl} \neq 0$ for some pair of distinct $\bk, \bl \in \Z_+^n$, then $S_p$ on $\clq_m$ is not liftable.
\end{corollary}

The following fact was used to prove the above theorem \cite[Corollary 2.3.]{Deepak}: For $\vp \in H^\infty(\T^n)$ with $\|\vp\|_\infty = 1$, if the Toeplitz operator $T_\vp$ is norm attaining, then the symbol $\vp$ is inner. In this context, it is worth noting that the lift of an operator in the commutant is highly nonunique, and the issue of uniqueness of Sarason's commutant lifting theorem is inextricably linked to the norm attaintment property \cite[Section 5]{Sarason}.

Now we consider a simple class of quotient modules where all module maps admit lifting. Our idea is fairly elementary: embed one Hardy space into another Hardy space. Fix a natural number $m$ such that $1 < m < n$. Define
\[
\cls = \sum_{j=1}^{m} z_j H^2(\T^n).
\]
Then $\cls$ is a closed subspace \cite{Sarkar 2015}. As
\[
z_i \cls \subseteq \cls,
\]
for all $i=1, \ldots, n$, it follows that $\cls$ is a submodule of $H^2(\T^n)$ . Our interest is in the corresponding quotient module $\clq$, that is
\[
\clq:= \Big(\sum_{j=1}^{m} z_j H^2(\T^n)\Big)^\perp.
\]
A simple calculation reveals that
\[
\clq = \mathbb{C} \otimes H^2(\T^{n-m}),
\]
where $\mathbb{C}$ denotes the subspace of all constant functions in $H^2(\mathbb{T}^m)$. In other words, $\clq$ is simply the space of functions on $H^2(\T^n)$ that does not depend on the variables $\{z_1, \ldots, z_m\}$ (again, see \cite{Sarkar 2015}). Because $S_{z_i} = P_\clq T_{z_i}|_\clq$, we have
\[
S_{z_i} =
\begin{cases}
0 & \mbox{if } i=1, \ldots, m
\\
T_{z_i} & \mbox{if } i = m+1, \ldots, n.
\end{cases}
\]
Let $X \in \clb(\clq)$. Then, by a routine argument, $X$ is a module map, that is
\[
X P_{\clq} T_{z_i}|_{\clq} = P_{\clq} T_{z_i}|_{\clq} X,
\]
for all $i=1, \ldots, n$, if and only if there exists $\vp \in H^\infty(\D^n)$ such that $\vp$ does not depend on the variables $\{z_1, \ldots, z_m\}$ and
\[
X = T_\vp.
\]
This immediately implies the following result: Let $X \in \clb_1(\clq)$ be a module map. Then $X = T_\vp$ for some $\vp \in \cls(\D^n)$. In particular, $X$ lifts to $T_\vp$ itself.

In Section \ref{subsect: example verification}, we will show examples of module maps on nonhomogeneous quotient modules that cannot be lifted.

\section{Classifications of commutant lifting}\label{sec: CLT Classification}

Given the examples in the preceding section, it is clear that a module map on a quotient module of $H^2(\T^n)$, $n \geq 2$, may not admit a lift in general. In this section, we classify liftable module maps defined on quotient modules of $H^2(\T^n)$, $n \geq 1$. We begin with the well known duality of classical Banach spaces. Recall that $L^1(\T^n)$ is a Banach space predual of $L^\infty(\T^n)$. More specifically, we have
\[
(L^1(\T^n))^* \cong L^\infty(\T^n),
\]
via the isometrically isomorphic map $\chi: L^\infty(\T^n) \rightarrow (L^1(\T^n))^*$ defined by $\vp \in L^\infty(\T^n) \mapsto \chi_\vp \in (L^1(\T^n))^*$, where for each $\vp \in L^\infty(\T^n)$, $\chi_\vp \in (L^1(\T^n))^*$ is defined by
\begin{equation}\label{eqn: duality}
\chi_\vp f = \int_{\T^n} \vp f \; d\mu,
\end{equation}
for all $f \in L^1(\T^n)$. Moreover, we have the isometric property
\[
\|\chi_\vp\| = \|\vp\|_\infty,
\]
for all $\vp \in L^\infty(\T^n)$. For a nonempty $X \subseteq L^2(\T^n)$, we define
\[
X^{conj} = \{\bar{f}: f \in X\}.
\]
We also define the subspace of ``mixed functions'' of $L^2(\T^n)$ as
\[
\clm_n = L^2(\T^n) \ominus ({H^2(\T^n)}^{conj} + H^2(\T^n)).
\]
This is the closed subspace of $L^2(\T^n)$ generated by monomials that are neither analytic nor coanalytic. Let $I_n = \{1, \ldots, n\}$. Given $A \subseteq I_n$, we set
\[
|A| = \# A,
\]
the cardinality of $A$. The following easy-to-see equality explains the terminology of ``mixed functions'':
\begin{equation}\label{eqn: def of Mn}
\clm_n = \overline{\text{span}} \{z_{A}^{k_A} \bar{z}_B^{k_B}: A,B \subseteq I_n, A \cap B = \emptyset, A, B \neq \emptyset, k_A \in \Z_+^{|A|}, k_B \in \Z_+^{|B|}\},
\end{equation}
where for a nonempty subset $A = \{i_1, \ldots, i_m\} \subsetneqq I_n$ and $k_A = (k_1, \ldots, k_m) \in \Z_+^{|A|}$, we define the monomial
\[
z_{A}^{k_A} := z_{i_1}^{k_1} \cdots z_{i_m}^{k_m}.
\]
Note that $\clm_n$ is self-adjoint, that is
\begin{equation}\label{eqn: conj self adj}
{\clm}^{conj}_n = \clm_n.
\end{equation}
It is also crucial to observe that if $n=1$, then $\clm_n$ is trivial:
\begin{equation}\label{eqn: M1 = 0}
\clm_1 = \{0\}.
\end{equation}
Given a quotient module $\clq \subseteq H^2(\T^n)$, as per our convention, we have ${\clq}^{conj}:= \{\bar{f}: f \in \clq\}$, and hence $\clq^{conj}$ is a closed subspace of $L^2(\T^n)$ and
\[
\clq^{conj} \perp H^2_0(\T^n),
\]
where $H^2_0(\T^n) = H^2(\T^n) \ominus \{1\}$. It is easy to check (for instance, by using $\bS(\cdot, 0) \equiv 1$) that
\[
H^2_0(\T^n) = \{f \in H^2(\T^n): f(0) = 0\},
\]
the closed subspace of $H^2(\T^n)$ of functions vanishing at the origin. Finally, given a quotient module $\clq \subseteq H^2(\T^n)$, we set
\[
\clm_\clq = \clq^{conj} \dotplus (\clm_n \dotplus H^2_0(\T^n)).
\]
The skew sums in the above definition are in fact Hilbert space orthogonal direct sums in $L^2(\T^n)$. However, in what follows, we will represent $\clm_\clq$ as a linear subspace of the Banach space $L^1(\T^n)$, and denote it by
\[
(\clm_\clq, \|\cdot\|_1).
\]
We are now ready for our first lifting theorem.

\begin{theorem}\label{thm: new CLT 1}
Let $\clq \subseteq H^2(\T^n)$ be a quotient module, $X \in \clb_1(\clq)$ be a module map, and let $\psi = X(P_{\clq} 1)$. Define  $X_{\clq}: (\clm_\clq, \|\cdot\|_1) \longrightarrow \mathbb{C}$ by
\[
X_\clq f = \int_{\T^n} \psi f \;d\mu,
\]
for all $f \in \clm_\clq$. Then $X$ is liftable if and only if $X_\clq$ is a contractive functional on $(\clm_\clq, \|\cdot\|_1)$.
\end{theorem}
\begin{proof}
Let $\vp \in \cls(\D^n)$ be a lift of $X$. Then $X = S_\vp$, where, by definition, $S_\vp = P_{\clq} T_\vp|_{\clq}$. Since $\vp \in \cls(\D^n)$ (that is, $\|\vp\|_\infty \leq 1$), it follows that the functional $\chi_\vp: L^1(\T^n) \raro \mathbb{C}$ defined by
\[
\chi_\vp(f) = \int_{\T^n} f \vp \,d\mu,
\]
for all $f \in L^1(\T^n)$, is a contraction (see \eqref{eqn: duality}). In view of the fact that $\vp \clq^\perp \subseteq \clq^\perp$ (as submodules are invariant under $H^\infty(\D^n)$), we have $P_\clq T_\vp P_\clq = P_\clq T_\vp$, and hence
\[
\begin{split}
S_\vp P_\clq 1 & = P_\clq T_\vp|_{\clq} P_\clq 1
\\
& = P_\clq T_\vp 1
\\
& = P_\clq \vp.
\end{split}
\]
Also, $X = S_\vp$ implies $\psi = S_\vp P_\clq 1$. This combined with $S_\vp P_\clq 1 = P_\clq \vp$ yields
\[
\psi = P_\clq \vp.
\]
We now prove that $X_\clq$ on $(\clm_\clq, \|\cdot\|_1)$ is a contractive functional by showing that
\[
X_\clq = \chi_\vp,
\]
on $\clm_\clq$. First we consider $X_\clq$ on ${\clq}^{conj} \subseteq \clm_\clq$. Let $\bar h \in {\clq}^{conj}$. Then $h \in \clq$ or, equivalently, $P_\clq h = h$, and we have
\[
\begin{split}
\int_{\T^n} \vp \bar{h} \, d\mu & = \langle \vp, h \rangle_{H^2(\T^n)}
\\
& = \langle \vp, P_\clq h \rangle_{H^2(\T^n)}
\\
& = \langle P_\clq \vp, h \rangle_{H^2(\T^n)}
\\
& = \langle \psi, h \rangle_{H^2(\T^n)}.
\end{split}
\]
Thus we conclude that
\[
\int_{\T^n} \psi \bar{h} \, d\mu  = \int_{\T^n} \vp \bar{h} \, d\mu,
\]
for all $\bar h \in {\clq}^{conj}$, equivalently
\[
X_\clq = \chi_\vp \text{ on } {\clq}^{conj}.
\]
Next, we consider $X_\clq$ on $\clm_n$. Since
\[
\clm_n \subseteq L^2(\T^n) \ominus (H^2(\T^n) + H^2(\T^n)^{conj}),
\]
functions in $\clm_n$  do not have an analytic part. Moreover, since $\clm_n$ is self-adjoint (see \eqref{eqn: conj self adj}), we have
\[
P_\clq {\clm}^{conj}_n = P_\clq \clm_n = \{0\}.
\]
By using the identity $\psi = P_\clq \vp$ and following the computation as in the previous case, for each $h \in \clm_n$, we have
\[
\begin{split}
\int_{\T^n} \psi h \, d\mu & = \langle \psi, \bar h \rangle_{L^2(\T^n)}
\\
& =  \langle P_\clq \vp, \bar h \rangle_{L^2(\T^n)}
\\
& =   \langle \vp, P_\clq \bar h \rangle_{H^2(\T^n)}
\\
& =  0,
\end{split}
\]
as $P_\clq h = P_\clq \bar h = 0$. This proves that
\[
X_\clq = \chi_\vp = 0 \text{ on } \clm_n.
\]
Finally, if $h \in H^2_0(\T^n)$, then $h(0) = 0$, and hence (as $\psi \in \clq \subseteq H^2(\T^n)$)
\[
\langle \bar{h}, \psi \rangle_{L^2(\T^n)} = 0.
\]
Therefore, again
\[
\begin{split}
\int_{T^n} \psi h \, d\mu & = \langle \psi, \bar{h} \rangle_{L^2(\T^n)}
\\
& = 0,
\end{split}
\]
as $\psi \in \clq \subset H^2(\T^n)$. This implies, again, that
\[
X_\clq = \chi_\vp = 0 \text{ on } H^2_0(\T^n).
\]
Thus we conclude that $X_\clq = \chi_\vp$ on $\clm_\clq$. On the other hand, $\chi_\vp: L^1(\T^n) \raro \mathbb{C}$ is a contraction. In particular, $\chi_\vp|_{\clm_\clq}$ is a contraction, which proves our claim that $X_\clq : \clm_\clq \rightarrow \mathbb{C}$ is contractive.

For the converse direction, assume that $X_\clq : (\clm_\clq, \|\cdot\|_1) \rightarrow \mathbb{C}$ is a contraction. By the Hahn-Banach theorem, there is a linear functional $\tilde{X}_\clq: L^1(\T^n) \longrightarrow \mathbb{C}$ such that
\[
\tilde{X}_\clq|_{\clm_\clq} = X_\clq,
\]
and
\[
\|\tilde{X}_\clq\| = \|X_\clq\| \leq 1.
\]
By the duality $(L^1(\T^n))^* \cong L^\infty(\T^n)$, as outlined in \eqref{eqn: duality}, there exists $\vp \in L^\infty(\T^n)$ such that
\[
\chi_\vp = \tilde{X}_\clq,
\]
and
\[
\|\vp\|_\infty \leq 1.
\]
In particular, $\chi_\vp|_{\clm_\clq} = \tilde{X}_\clq|_{\clm_\clq} = X_\clq$. Since
\[
\chi_\vp h = \int_{\T^n} \vp h \, d \mu,
\]
for all $h \in \clm_\clq$, it follows that
\begin{equation}\label{eqn: proof CLT 1}
\int_{\T^n} \vp h \, d\mu  = \int_{\T^n} \psi h \, d\mu,
\end{equation}
for all $h \in \clm_\clq$. We consider a typical monomial $f$ from $\clm_n \dotplus H^2_0(\T^n)$. In other words, we let
\[
f = z^k,
\]
for some $k\in \mathbb{N}^n$, or let
\[
f = z_A^{k_A} \bar{z}_B^{k_B},
\]
for some $k_A \in \Z_+^{|A|}$ and $k_B \in \Z_+^{|B|}$, where $A, B \subseteq \{1, \ldots, n\}$, $A \cap B = \emptyset$, and $A, B \neq \emptyset$ (see the definition of $\clm_n$ in \eqref{eqn: def of Mn}). As $\psi = X(P_{\clq} 1) \in \clq \subseteq \text{Hol}(\D^n)$, it follows that $\langle \psi, \bar{f} \rangle_{L^2(\T^n)} = 0$ and hence
\[
\int_{\T^n} \psi f \, d\mu = 0.
\]
Consequently, the identity in \eqref{eqn: proof CLT 1} yields
\[
\int_{\T^n} \vp z^{\bk} \, d\mu = 0,
\]
for all $k\in \mathbb{N}^n$, as well as
\[
\int_{\T^n} \vp z_A^{k_A} \bar{z}_B^{k_B} \, d\mu = 0,
\]
for all $k_A \in \Z_+^{|A|}$ and $k_B \in \Z_+^{|B|}$, where $A, B \subseteq \{1, \ldots, n\}$, $A \cap B = \emptyset$, and $A, B \neq \emptyset$. This implies $\vp$ is analytic, and hence $\vp \in \cls(\D^n)$. To complete the proof, it remains to show that $X = S_\vp$. Note, by \eqref{eqn: proof CLT 1} again, we have that
\[
\int_{\T^n} \psi \bar{h} \, d\mu = \int_{\T^n} \vp \bar{h} \, d\mu,
\]
for all $\bar h \in {\clq}^{conj}$. Equivalently, for each $\bar h \in {\clq}^{conj}$, we have
\[
\langle \vp, h\rangle_{L^2(\T^)} = \langle \psi, h\rangle_{L^2(\T^n)},
\]
and hence
\[
\langle P_\clq \vp, h \rangle_{H^2(\T^n)} = \langle \psi, h \rangle_{H^2(\T^n)},
\]
from which we conclude that
\[
P_\clq \vp = \psi.
\]
As before, we write $\vp \in \cls(\D^n) \subseteq H^2(\T^n)$ with respect to $\clq \oplus \clq^\perp = H^2(\T^n)$ as
\[
\vp = \psi \oplus \rho \in \clq \oplus \clq^\perp.
\]
Since $P_\clq \vp = P_\clq T_\vp P_\clq 1 = S_\vp (P_\clq 1)$ and $P_\clq \vp = \psi$, we have
\[
\psi = S_\vp (P_\clq 1).
\]
This combined with $\psi =  X(P_\clq 1)$ yields
\[
S_\vp (P_\clq 1) = X (P_\clq 1).
\]
Finally, let us fix $\bk \in \Z_+^n$ and observe
\[
\begin{split}
P_\clq z^{\bk} & = P_\clq z^{\bk} (P_\clq 1)
\\
& = S_z^{\bk} (P_\clq 1).
\end{split}
\]
Therefore, $S_\vp S_z^{\bk} = S_z^{\bk} S_\vp$ implies
\[
\begin{split}
S_\vp (P_\clq z^{\bk}) & = S_\vp S_z^{\bk} (P_\clq 1)
\\
& = S_z^{\bk} S_\vp (P_\clq 1)
\\
& = S_z^{\bk} X (P_\clq 1)
\\
& = X S_z^{\bk} (P_\clq 1)
\\
& = X (P_\clq z^{\bk}).
\end{split}
\]
Then, in view of the fact that $\clq = \overline{\text{span}} \{P_{\clq} z^{\bk}: \bk \in \Z_+^n\}$, the equality $X = S_\vp$ is immediate. This completes the proof of the theorem.
\end{proof}

The proof of the above theorem says more than what it states. In fact, we have the identity
\[
X_\clq|_{\clm_n \dotplus H^2_0(\T^n)} \equiv 0,
\]
and hence
\[
\ker X_\clq \supseteq \clm_n \dotplus H^2_0(\T^n).
\]
Another way to put it is that there is a contractive extension of $X_\clq|_{{\clq}^{conj}}$ to the entire $\clm_\clq$ that vanishes on the completely analytic and completely co-analytic parts.

\begin{remark}\label{rem: indept of X}
It is clear from the construction that the subspace $(\clm_{\clq}, \|\cdot\|_1)$ is independent of $X$.
\end{remark}

Our second lifting theorem is a consequence of the first, and it is in a more compact form. Given a quotient module $\clq \subseteq H^2(\T^n)$ and a module map $X \in \clb(\clq)$, we define a subspace of $L^1(\T^n)$ as
\[
\widetilde{\clm}_{\clq, X} = ({\clq}^{conj} \ominus \{\bar{\psi}\}) \dotplus (\clm_n \dotplus H^2_0(\T^n)),
\]
where
\[
\psi = X(P_\clq 1).
\]
Keep in mind, in contrast to Remark \ref{rem: indept of X}, that $(\widetilde{\clm}_{\clq, X}, \|\cdot\|_1)$ is dependent on $X$.

\begin{theorem}\label{thm: new CLT 2}
Let $\clq \subseteq H^2(\T^n)$ be a quotient module, $X \in \clb_1(\clq)$ be a module map, and let $\psi = X(P_\clq 1)$. Then $X$ is liftable if and only if
\[
\text{dist}_{L^1(\T^n)}\Big(\frac{\bar \psi}{\|\psi\|_2^2}, \widetilde{\clm}_{\clq, X}\Big) \geq 1.
\]
\end{theorem}
\begin{proof}
In view of $\bar{\psi} \in {\clq}^{conj}$, first we observe that
\[
\clm_\clq = \mathbb{C}\bar{\psi} \dotplus \widetilde{\clm}_{\clq, X}.
\]
Suppose $X$ is liftable. By Theorem \ref{thm: new CLT 1}, we have
\begin{equation}\label{eqn: proof 2nd lift 1}
\Big|\int_{\T^n} \psi f \, d\mu  \Big| \leq \|f\|_1 \qquad (f \in \clm_\clq).
\end{equation}
Pick $\tilde{g} \in \widetilde{\clm}_{\clq, X}$ and a scalar $c$. Define $g \in \clm_\clq$ by
\[
g = c \bar{\psi} + \tilde{g}.
\]
We compute
\[
\begin{split}
\int_{\T^n} \psi (c \bar{\psi} + \tilde{g}) \, d\mu & = c \int_{\T^n} \psi \bar{\psi} \, d\mu + \int_{\T^n} \psi \tilde{g} \, d\mu
\\
& = c \|\psi\|^2_2 + \langle \psi, \overline{\tilde{g}} \rangle
\\
& = c \|\psi\|^2_2,
\end{split}
\]
as
\[
\langle \psi, \overline{\tilde{g}} \rangle = 0,
\]
which follows from the fact that $\tilde{g} \perp \bar{\psi}$. Now \eqref{eqn: proof 2nd lift 1} implies
\[
\Big| \int_{\T^n} \psi (c \bar{\psi} + \tilde{g}) \, d\mu \Big| \leq \|c \bar{\psi} + \tilde{g}\|_1,
\]
and hence
\[
|c| \|\psi\|^2_2 \leq \|c \bar{\psi} + \tilde{g}\|_1,
\]
or equivalently
\[
\Big\|\frac{\bar{\psi}}{\|\psi\|^2_2} + \tilde{g}\Big\|_1 \geq 1,
\]
for all $\tilde{g} \in \widetilde{\clm}_{\clq, X}$, and completes the proof of the forward direction.

To prove the reverse direction, let the above inequality hold for all $\tilde{g} \in \widetilde{\clm}_{\clq, X}$. Equivalently
\[
\|\psi\|^2_2 \leq \|\bar{\psi} + \tilde{g}\|_1 \qquad (\tilde{g} \in \widetilde{\clm}_{\clq, X}).
\]
Fix $f \in \clm_\clq$, and write $f = c \bar{\psi} + \tilde{f}$ for some scalar $c$ and some function $\tilde{f} \in \widetilde{\clm}_{\clq, X}$. Following the proof of the forward direction, we have
\[
\begin{split}
c \|\psi\|^2_2 & = \int_{\T^n} \psi (c \bar{\psi} + \tilde{f}) \, d\mu
\\
& = \int_{\T^n} \psi f \, d\mu,
\end{split}
\]
which leads to \eqref{eqn: proof 2nd lift 1}. Theorem \ref{thm: new CLT 1} now completes the proof of the theorem.
\end{proof}

Combining Theorem \ref{thm: new CLT 1} and Theorem \ref{thm: new CLT 2}, we have the following:

\begin{theorem}\label{thm: combine CLT}
Let $\clq \subseteq H^2(\T^n)$ be a quotient module, and let $X \in \clb_1(\clq)$ be a module map. Set
\[
\psi = X(P_\clq 1),
\]
and
\[
\clm_\clq = {\clq}^{conj} \dotplus (\clm_n \dotplus H^2_0(\T^n)),
\]
and
\[
\widetilde{\clm}_{\clq, X} = ({\clq}^{conj} \ominus \{\bar{\psi}\}) \dotplus (\clm_n \dotplus H^2_0(\T^n)).
\]
Then the following conditions are equivalent:
\begin{enumerate}
\item $X$ is liftable.
\item $X_{\clq}: (\clm_\clq, \|\cdot\|_1) \longrightarrow \mathbb{C}$ is a contractive functional, where
\[
X_{\clq} f = \int_{\T^n} \psi f \;d\mu \qquad (f \in \clm_\clq).
\]
\item $\text{dist}_{L^1(\T^n)}\Big(\frac{\bar \psi}{\|\psi\|_2^2}, \widetilde{\clm}_{\clq, X}\Big) \geq 1$.
\end{enumerate}
\end{theorem}

The techniques involved in the association of the existence of commutant lifting with the distance formula are far-reaching. In the following section, we will apply some of the concepts introduced here to solve a perturbation problem.

\section{Perturbations of analytic functions}\label{sec: perturbation}

Our aim is to present a classification of $H^2(\T^n)$-functions that can be perturbed by $H^2(\T^n)$-functions so that the resultant functions are in $\cls(\D^n)$. Our perturbation result is of independent interest and not directly related to the commutant lifting theorem. However, the technique involved here is motivated by the one used in the proof of our main results.

Throughout the sequel, we denote
\[
\cll_n = \clm_n \oplus H^2_0(\T^n).
\]
Recall that $H^2_0(\T^n) = H^2(\T^n) \ominus \{1\}$ is the closed subspace of $H^2(\T^n)$ of functions vanishing at the origin. Recall also that
\[
\clm_n = L^2(\T^n) \ominus ({H^2(\T^n)}^{conj} + H^2(\T^n)),
\]
the closed subspace of $L^2(\T^n)$ generated by all the trigonometric monomials that are neither analytic nor co-analytic. In particular, we have the crucial property that
\[
\langle f, 1 \rangle_{L^2(\T^n)} = 0 \qquad (f \in \cll_n).
\]
Finally, we recall a basic fact from Banach space theory: Let $x$ be a vector in a Banach space $B$. Then
\[
\|x\|_B = \sup \{|x^*(x)|: x^* \in B^*, \|x^*\| \leq 1\}.
\]
Now we are ready for the perturbation theorem.

\begin{theorem}\label{thm: perturbation}
Let $f \in H^2(\T^n)$ be a nonzero function. There exists $g \in \{f\}^\perp$ such that $f + g \in \cls(\D^n)$ if and only if
\[
\text{dist}_{L^1(\T^n)}\Big(\frac{\bar f}{\|f\|_2^2}, \cll_n \Big) \geq 1.
\]
\end{theorem}
\begin{proof}
We start by recalling the definition of distance function:
\[
\text{dist}_{L^1(\T^n)}\Big(\frac{\bar f}{\|f\|_2^2}, \cll_n \Big) = \inf \Big\{\Big\|\frac{\bar f}{\|f\|_2^2} + h\Big\|_1: h \in \cll_n\Big\}.
\]
Suppose $g \in \{f\}^\perp$ be such that $\psi:= f + g \in \cls(\D^n)$. It is enough to prove that
\[
\Big\|\frac{\bar f}{\|f\|_2^2} + h\Big\|_1 \geq 1 \qquad (h \in \cll_n).
\]
Fix $h \in \cll_n$. Since $\psi \in \cls(\D^n)$ and $\cls(\D^n)$ is a subset of the closed unit ball of $L^\infty(\T^n)$, we have $\psi \in L^\infty(\T^n)$ and $\|\psi\|_\infty \leq 1$. By the duality (see \eqref{eqn: duality})
\[
(L^1(\T^n))^* \cong L^\infty(\T^n),
\]
it follows that $\chi_\psi \in (L^1(\T^n))^*$ and
\[
\|\psi\|_\infty = \|\chi_\psi\| \leq 1,
\]
where
\[
\chi_\psi g = \int_{\T^n} \psi g \,d\mu,
\]
for all $g \in L^1(\T^n)$. In particular, for
\[
g = \frac{\bar f}{\|f\|_2^2} + h \in L^1(\T^n),
\]
we compute
\[
\begin{split}
\int_{\T^n} \psi \Big(\frac{\bar f}{\|f\|_2^2} + h\Big) \,d\mu & = \Big\langle f + g, \frac{f}{\|f\|_2^2} + \bar{h} \Big\rangle_{L^2(\T^n)}
\\
& = 1 + \Big\langle g, \frac{f}{\|f\|_2^2} + \bar{h} \Big\rangle_{L^2(\T^n)}
\\
& = 1.
\end{split}
\]
The last but one equality follows from the fact that (note that $\bar h$ has no analytic part)
\[
\langle f, \bar{h} \rangle_{L^2(\T^n)} = 0,
\]
and the last equality is due to the fact that $g \in \{f\}^\perp$ and
\[
\langle g, \bar{h} \rangle_{L^2(\T^n)} = 0,
\]
similar reason as in the preceding equality. We also have used the fact that $f$ is analytic and $\langle h, 1 \rangle_{L^2(\T^n)} = 0$. Therefore, $\chi_\psi \in (L^1(\T^n))^*$ with $\|\chi_\psi\| \leq 1$ and
\[
\Big|\chi_\psi \Big(\frac{\bar f}{\|f\|_2^2} + h\Big) \Big| = 1.
\]
The norm identity for Banach spaces stated preceding the statement of this theorem immediately implies that
\[
\Big\|\frac{\bar f}{\|f\|_2^2} + h\Big\|_1 \geq 1.
\]
For the reverse direction, suppose the above inequality holds for all $h \in \cll_n$. Equivalently
\[
\|\lambda \bar{f} + h\|_1 \geq |\lambda | \|f\|_2^2,
\]
for all $h \in \cll_n$ and $\lambda \in \mathbb{C}$. Define $\cls$ a subspace of $L^1(\T^n)$ as
\[
\cls:= \text{span} \{\bar f, \cll_n\},
\]
and then define a linear functional $\zeta_f: \cls \raro \mathbb{C}$ by
\[
\zeta_f (\lambda \bar f + h) = \int_{\T^n} (\lambda \bar{f} + h) f \, d\mu,
\]
for all $h \in \cll_n$ and $\lambda \in \mathbb{C}$. As in the proof of the forward direction, we have
\[
\begin{split}
\int_{\T^n} f h \, d\mu & = \langle h, \bar f \rangle_{L^2(\T^n)}
\\
& = 0,
\end{split}
\]
for all $h \in \cll_n$. Moreover, since
\[
\int_{\T^n} f \bar f \, d\mu = \|f\|^2_2,
\]
it follows that
\[
\begin{split}
|\zeta_f (\lambda \bar f + h)| & = |\lambda| \| f \|^2_2
\\
& \leq \|\lambda \bar f + h\|_1,
\end{split}
\]
for all $h \in \cll_n$ and $\lambda \in \mathbb{C}$. This ensures that $\zeta_f$ is a contractive functional on $\cls$; hence, by the Hahn-Banach theorem, there exists $\zeta \in (L^\infty(\T^n))^*$ such that $\|\zeta\| \leq 1$ and
\[
\zeta|_{\cls} = \zeta_f.
\]
Again, by the duality \eqref{eqn: duality}, there exists $\vp \in L^\infty(\T^n)$ such that $\|\vp\|_\infty \leq 1$ and
\[
\chi_\vp|_{\cls} = \zeta|_{\cls} = \zeta_f.
\]
Therefore
\begin{equation}\label{eqn: proof pert 1}
\int_{\T^n} (\lambda \bar f + h) f \, d\mu = \int_{\T^n} (\lambda \bar f + h) \vp \, d\mu,
\end{equation}
for all $h \in \cll_n$ and $\lambda \in \mathbb{C}$. We now claim that $\vp$ is analytic (which would clearly imply that $\vp \in H^\infty(\D^n)$). As in the proof of Theorem \ref{thm: new CLT 1}, we consider a typical monomial $F$ from $\cll_n = \clm_n \dotplus H^2_0(\T^n)$. Therefore
\[
F = z^k,
\]
for some $k\in \mathbb{N}^n$, or
\[
F = z_A^{k_A} \bar{z}_B^{k_B},
\]
for some $k_A \in \Z_+^{|A|}$ and $k_B \in \Z_+^{|B|}$, where $A, B \subseteq \{1, \ldots, n\}$, $A \cap B = \emptyset$, and $A, B \neq \emptyset$  (see the representation of $\clm_n$ in \eqref{eqn: def of Mn}). We compute
\[
\begin{split}
0 & = \langle f, \bar{F} \rangle_{L^2(\T^n)}
\\
& = \int_{\T^n} f F \,d\mu
\\
& = \int_{\T^n} \vp F \,d\mu
\\
& = \langle \vp, \bar{F} \rangle_{L^2(\T^n)},
\end{split}
\]
which proves the claim. Since $\|\vp\|_\infty \leq 1$, we conclude that $\vp \in \cls(\D^n)$. Using the containment $H^\infty(\D^n) \subseteq H^2(\D^n)$, first we conclude $\vp \in H^2(\D^n)$, and then write
\[
\vp = c f \oplus g,
\]
for some scalar $c$ and function $g \in H^2(\D^n)$ such that $g \in \{f\}^\perp$. It remains to show that $c = 1$. Observe, if $h = 0$, and
\[
\lambda = \frac{1}{\|f\|_2^2},
\]
then \eqref{eqn: proof pert 1} along with the fact that $\langle g, f \rangle = 0$ yields
\[
\begin{split}
1 & = \int_{\T^n} f \frac{\bar f}{\|f\|_2^2} \, d\mu
\\
& = \int_{\T^n} \vp \frac{\bar f}{\|f\|_2^2} \, d\mu
\\
& = \Big\langle \vp, \frac{f}{\|f\|_2^2} \Big\rangle_{H^2(\T^n)}
\\
& = \Big\langle c f \oplus g, \frac{f}{\|f\|_2^2} \Big\rangle_{H^2(\T^n)}
\\
& = c.
\end{split}
\]
This completes the proof of the theorem.
\end{proof}

We know, in particular, that $\clm_1 = \{0\}$ (see \eqref{eqn: M1 = 0}). Moreover, as observed earlier, that $H^2_0(\T) = z H^2(\T)$. Therefore
\[
\cll_1 = z H^2(\T),
\]
and as a result, the preceding theorem is simplified as follows:

\begin{corollary}
Given $f \neq 0$ in $H^2(\T)$, there exists $g \in \{f\}^\perp$ such that
\[
f + g \in \cls(\D),
\]
if and only if
\[
\text{dist}_{L^1(\T)}\Big(\frac{\bar f}{\|f\|_2^2}, z H^2(\T)\Big) \geq 1.
\]
\end{corollary}

The following example illustrates the above theorem.

\begin{example}
Fix a real number $0 < c <1$, and pick $b \in (c^2, c)$. Also fix a multiindex $k_0 \in \Z_+^n$, $k_0 \neq (0, \ldots, 0)$, and set
\[
\Lambda := \Z_+^n \setminus \{k_0\}.
\]
Finally, choose a sequence $\{a_k\}_{k \in \Lambda} \subseteq \mathbb{R}_+$ such that
\begin{enumerate}
\item $\displaystyle\sum_{k \in \Lambda} a_k$ diverges, and
\item $\displaystyle\sqrt{\sum_{k \in \Lambda} a_k^2 + b^2} = c$.
\end{enumerate}
Set
\[
f = \sum_{k \in \Lambda} a_k z^k + b z^{k_0}.
\]
We want to show that $f$ can be perturbed to become a Schur function. To this end, we first observe that $f(0, \ldots, 0)= a_0$ and
\[
f(1, \ldots, 1) = b + \sum_{k \in \Lambda} a_k,
\]
and hence (by continuity)
\[
f(L) = (a_0, \infty),
\]
where $L$ is the line joining $(0, \ldots, 0)$ and $(1, \ldots, 1)$. We conclude, in particular, that
\[
f \notin H^\infty(\D^n).
\]
Moreover
\[
\|f\|_2 = c,
\]
by construction of $f$. We now consider the functional $\chi_{z^{k_0}} \in (L^1(\T^n))^*$ (see the duality \eqref{eqn: duality}). Clearly
\[
\|\chi_{z^{k_0}}\| = 1.
\]
Given arbitrary functions $g \in \clm_n$ and $h \in H^2_0(\T^n)$, we compute
\[
\begin{split}
\chi_{z^{k_0}} \Big(\frac{\bar{f}}{c^2} + g + h\Big) & = \int_{\T^n} z^{k_0} \Big(\frac{\bar{f}}{c^2} + g + h\Big) d\mu\\
& = \Big\langle \frac{\bar{f}}{c^2} + g + h, \bar{z}^{k_0} \Big\rangle_{L^2(\T^n)}
\\
& = \Big\langle \frac{\bar{f}}{c^2}, \bar{z}^{k_0} \Big\rangle_{L^2(\T^n)}
\\
& = \frac{b}{c^2}.
\end{split}
\]
Since $b > c^2$, it follows that
\[
\chi_{z^{k_0}} \Big(\frac{\bar{f}}{c^2} + g + h\Big) \geq 1,
\]
and consequently, the norm identity that was mentioned preceding the statement of Theorem \ref{thm: interpolation} infers that
\[
\Big\|\frac{\bar{f}}{c^2} + g + h\Big\|_1 \geq 1.
\]
Given that $\cll_n = \clm_n \dotplus H^2_0(\T^n)$, the above is equivalent to saying that
\[
\text{dist}_{L^1(\T^n)}\Big(\frac{\bar f}{\|f\|_2^2}, \cll_n \Big) \geq 1,
\]
and hence, by Theorem \ref{thm: perturbation}, we conclude that $f \oplus g \in \cls(\D^n)$ for some $g \in \{f\}^\perp$.
\end{example}

It may appear to be a coincidence that the distance recipe in Theorem \ref{thm: perturbation} as well as in Theorem \ref{thm: new CLT 2} (and the quantitative interpolation theorem in Section \ref{sec: quant inter}) is similar to the well known Nehari theorem \cite{N} for Hankel operators. Recall that the Hankel operator with symbol $\vp \in L^\infty(\T)$ is defined by
\[
H_\vp = P_{H^2_-(\T)} L_\vp|_{H^2(\T)},
\]
where $H^2_-(\T) = L^2(\T) \ominus H^2(\T)$. The Nehari theorem states:
\[
\|H_\vp\| = \text{dist}(\vp, H^\infty(\D)) = \|\vp\|_\infty.
\]
Furthermore, it is well known that the Nehari problem is related to the Nevanlinna-Pick interpolation problem for rational functions. See also the well known Adamyan, Arov, and Krein theorem, also known as the AAK step-by-step extension \cite[Chapter 2]{Nik}. Another important formula is due to Adamyan, Arov and Krein \cite{AAK1}:\
\[
\|H_\vp\|_{ess} = \text{dist}(\vp, C(\T) + H^\infty(\D)),
\]
for all $\vp \in L^\infty(\T)$, where $C(\T)$ denotes the space of all continuous functions on $\T$, and $\|H_\vp\|_{ess}$ denotes the essential norm of $H_\vp$. Hankel operators in several variables \cite{Ro 98} also present significant challenges. We refer the reader to Coifman, Rochberg, and Weiss \cite{Richard} for some progress to the theory of Hankel operators (also see \cite{FS}).

\section{Interpolation}\label{sec: weak lift}

The goal of this section is to provide a solution to the interpolation problem. As previously mentioned, Sarason's commutant lifting theorem recovers the Nevanlinna-Pick interpolation with an elegant proof. However, Sarason only needed to use his lifting theorem for some special finite-dimensional quotient modules. These quotient modules are generated by finitely many kernel functions.

First, we prove that Sarason-type quotient modules (we call them zero-based quotient modules) in several variables always admit lifting to $H^\infty(\D^n)$-functions (we call it weak lifting).

\begin{definition}
Let $\clq \subseteq H^2(\D^n)$ be a quotient module, and let $X \in \clb(\clq)$. Suppose $X S_{z_i} = S_{z_i} X$ for all $i=1, \ldots, n$. We say that $X$ admits a weak lift or $X$ is weakly liftable if there exists $\vp \in H^\infty(\D^n)$ such that
\[
X = S_\vp.
\]
\end{definition}

To put it another way, a weak lifting is a lifting that lacks control over the norm. Given a set $\clz \subseteq \D^n$, define
\[
\clq_\clz = \overline{\text{span}} \{\mathbb{S}(\cdot, w): w \in \clz\}.
\]

\begin{definition}
A quotient module $\clq \subseteq H^2(\T^n)$ is said to be \textit{zero-based} if there exists $\clz \subseteq \D^n$ such that $\clq = \clq_\clz$.
\end{definition}

For a zero-based quotient module $\clq_{\clz}$, by using the reproducing property \eqref{eqn: rep prop}, we have the following representation of the corresponding submodule (hence the name zero-based)
\[
\clq_{\clz}^\perp = \{f \in H^2(\T^n): f(w) = 0 \text{ for all } w \in \clz\}.
\]
Since $\{\mathbb{S}(\cdot, w): w \in \clz\}$ is a set of linearly independent vectors, a zero-based quotient module $\clq_\clz$ is finite-dimensional if and only if
\[
\# \clz = \text{dim} \clq_\clz < \infty.
\]
For each $j \in \{1, \ldots, n\}$, denote by $\pi_j: \mathbb{C}^n \longrightarrow \mathbb{C}$ the projection map onto the $j$-th coordinate. In particular, $z \in \mathbb{C}^n$ can be expressed as
\[
z = (\pi_1(z), \ldots, \pi_n(z)).
\]
The following easy-to-see lemma will be useful in what follows.

\begin{lemma}\label{eqn: Y star S}
Let $\clz = \{z_i\}_{i=1}^m \subset \D^n$ be a set of distinct points, and let $X \in \clb(\clq_\clz)$. Then $X$ is module map if and only if there exists $\{w_i\}_{i=1}^m \subset \mathbb{C}$ such that
\[
X^* \mathbb{S} (\cdot, z_j) = {w}_j \mathbb{S} (\cdot, z_j),
\]
for all $j=1, \ldots, m$.
\end{lemma}
\begin{proof}
Let $X \in \clb(\clq_\clz)$ and suppose $X S_{z_i} = S_{z_i} X$ for all $i=1, \ldots, n$. Since $X^* S_{z_i}^* = S_{z_i}^* X^*$, using the fact that $\clq_\clz$ is a quotient module, we find
\[
T_{z_i}^*|_{\clq_\clz} X^* = X^* T_{z_i}^*|_{\clq_\clz},
\]
for all $i=1, \ldots, n$. In view of \eqref{eqn: eigen vec kernel 2}, we compute
\[
\begin{split}
(T_{z_i}^*|_{\clq_\clz} X^*) \mathbb{S} (\cdot, z_j) & = (X^* T_{z_i}^*|_{\clq_\clz})\mathbb{S} (\cdot, z_j)
\\
& = X^* T_{z_i}^* \mathbb{S} (\cdot, z_j)
\\
& = \overline{\pi_i(z_j)} X^* \mathbb{S} (\cdot, z_j).
\end{split}
\]
Since $(T_{z_i}^*|_{\clq_\clz} X^*) \mathbb{S} (\cdot, z_j) = T_{z_i}^*(X^* \mathbb{S}(\cdot, z_j))$, it follows that
\[
T_{z_i}^*(X^* \mathbb{S}(\cdot, z_j)) = \overline{\pi_i(z_j)} (X^* \mathbb{S}(\cdot, z_j)),
\]
for all $i = 1, \ldots, n$, and $j=1, \ldots, m$. Equivalently
\[
X^* \mathbb{S}(\cdot, z_j) \in \bigcap_{i=1}^n \ker (T_{z_i} - \pi_i(z_j) I_{H^2(\T^n)})^*,
\]
for all $j=1, \ldots, m$. Now, in view of the joint eigenspace property \eqref{eqn: eigen vec kernel}, the right side of the above is $\mathbb{C} \mathbb{S}(\cdot, z_j)$, and hence, there exists a scalar $w_j$ such that
\[
X^* \mathbb{S} (\cdot, z_j) = {w}_j \mathbb{S} (\cdot, z_j),
\]
for all $j=1, \ldots, m$. The converse direction is easy and follows again from \eqref{eqn: eigen vec kernel 2} and the definition of $\clq_\clz$.
\end{proof}

The proposition that follows is very crucial and will be used in what follows.

\begin{proposition}
\label{lemma: proj at 1}
Let $\clq \subseteq H^2(\T^n)$ be a quotient module. Let
\[
\theta_{\clq} = P_{\clq} 1.
\]
If $\theta_{\clq} \in H^\infty(\D^n)$, then $S_{\theta_{\clq}} = I_{\clq}$, the identity operator on $\clq$.
\end{proposition}
\begin{proof}
Since $\theta_{\clq} = P_{\clq} 1 \in \clq \cap H^\infty(\D^n)$, in view of the decomposition $H^2(\T^n) = \clq \oplus \clq^\perp$, there exists $\vp \in H^\infty(\D^n) \cap \clq^\perp$ such that
\[
1 = \theta_{\clq} \oplus \vp \in \clq \oplus \clq^\perp.
\]
Fix $f \in \clq$. In particular, since $f \in H^2(\T^n)$, there exists a sequence
\[
\{p_j\}_{j=1}^\infty \subseteq \mathbb{C}[z_1, \ldots, z_n],
\]
such that
\[
p_j \longrightarrow f \quad \text{ in }  H^2(\T^n).
\]
Since $\vp \in H^\infty(\D^n) \cap \clq^\perp$ is a multiplier, the above implies
\[
\vp p_j \longrightarrow \vp f \quad \text{ in }  H^2(\T^n).
\]
Moreover, $\vp \in \clq^\perp$ implies that
\[
\{p_j \vp\}_{j=1}^\infty \subseteq \clq^\perp,
\]
as $\clq^\perp$ is a submodule, and hence $\vp f \in \clq^\perp$. Equivalently, we have
\[
P_\clq (\vp f) = 0.
\]
Finally, since $\theta_{\clq}, \vp \in H^\infty(\D^n)$, it follows that
\[
\begin{split}
f & = \theta_{\clq} f + \vp f
\\
& = P_{\clq} (\theta_{\clq} f + \vp f ) \qquad (\text{as } f \in \clq)
\\
& = P_{\clq} (\theta_{\clq} f) + 0
\\
& = S_{\theta_{\clq}} f,
\end{split}
\]
which yields $S_{\theta_{\clq}} f = f$, and completes the proof of the proposition.
\end{proof}

We are now ready for the weak lifting. It asserts, in essence, that a module map on a finite-dimensional zero-based quotient module always admits a lift to $H^\infty(\D^n)$.

\begin{corollary}\label{cor: lift without 1}
Let $\clz = \{z_1, \ldots, z_m\} \subset \D^n$ be $m$ distinct points, $X \in \clb(\clq_\clz)$, and let
\[
\vp = X(P_{\clq_\clz} 1).
\]
Then
\[
X S_{z_i} = S_{z_i} X,
\]
for all $i=1, \ldots, n$, if and only if $\vp \in H^\infty(\D^n)$. Moreover, in this case, we have
\[
X = S_\vp,
\]
and, in particular
\[
\vp \in H^\infty(\D^n) \cap \clq_\clz.
\]
\end{corollary}
\begin{proof}
The sufficient part is trivial. We prove the necessary part. For simplicity of notation, we set $\clq = \clq_\clz$. Let $X \in \clb(\clq)$ and suppose that $X S_{z_i} = S_{z_i} X$ for all $i=1, \ldots, n$. As in Proposition \ref{lemma: proj at 1}, set
\[
\theta_{\clq} = P_\clq 1.
\]
As observed earlier, $\mathbb{S}(\cdot, w) \in H^\infty(\D^n)$ for all $w \in \D^n$ implies that $\clq \subseteq H^\infty(\D^n)$. In particular, $\theta_{\clq} \in H^\infty(\D^n)$. By Proposition \ref{lemma: proj at 1}, we have
\[
S_{\theta_\clq} = I_\clq.
\]
Since $X \in \clb(\clq)$, it follows that
\[
\vp:= X \theta_\clq \in H^\infty(\D^n).
\]
Therefore
\[
\begin{split}
S_\vp \theta_\clq &= P_\clq (\vp \theta_\clq)
\\
& = S_{\theta_\clq} \vp
\\
& = \vp
\\
& = X \theta_\clq .
\end{split}
\]
The remainder of the proof is based on the standard property of the module map $X$. Indeed, we first observe that
\[
\clq = \overline{\text{span}} \{P_{\clq} z^k P_{\clq} 1: k \in \Z_+^n\}.
\]
On the other hand, for $\bk \in \Z_+^n$, since $X S_z^{\bk} = S_z^{\bk}X$, we have
\[
\begin{split}
X (P_\clq (z^{\bk} \theta_\clq)) & = X(S_z^{\bk} \theta_\clq)
\\
& = S_z^{\bk} X \theta_\clq
\\
& = P_\clq z^{\bk} \vp
\\
& = P_\clq z^{\bk} S_{\vp} \theta_\clq
\\
& = S_\vp (P_\clq (z^{\bk} \theta_\clq)).
\end{split}
\]
This completes the proof of the fact that $X = S_\vp$. The final assertion follows from the definition of $\theta_{\clq}$.
\end{proof}

As already pointed out, weak lifting does not capture the delicate structure of Schur functions on $\D^n$, $n > 1$.

We will now look at the interpolation problem. Recall once again that
\[
\bS(z, w) = \prod_{i=1}^{m} \frac{1}{1 - z_i \bar{w}_i} \qquad (z, w \in \D^n),
\]
is the Szeg\"{o} kernel of $\D^n$, and
\[
\bS(z, w) = \Big\langle \bS(\cdot, w), \bS(\cdot, z) \Big\rangle_{H^2(\T^n)} \qquad (z, w \in \D^n).
\]

\begin{theorem}\label{thm: interpolation}
Let $\clz = \{z_i\}_{i=1}^m \subset \D^n$ be $m$ distinct points, and let $\{w_i\}_{i=1}^m \subset \D$ be $m$ scalars. Set
\[
\clm_{\clq_\clz} = {\clq}^{conj}_{\clz} \dotplus (\clm_n \dotplus H^2_0(\T^n)).
\]
Then there exists $\vp \in \cls(\D^n)$ such that
\[
\vp(z_i) = w_i,
\]
for all $i=1, \ldots, m$, if and only if
\[
\mi_{\clz, \clw} f = \int_{\T^n} \psi_{\clz, \clw} f \,d \mu \qquad (f \in \clm_{\clq_\clz}),
\]
defines a contraction $\mi_{\clz, \clw}: (\clm_{\clq_\clz}, \|\cdot\|_1) \raro \mathbb{C}$, where
\[
\psi_{\clz, \clw} = \sum_{i=1}^{m} c_i \bS(\cdot, z_i),
\]
and the scalar coefficients $\{c_i\}_{i=1}^m$ are given by the identity
\[
\begin{bmatrix}
c_1
\\
c_2\\
\vdots
\\
c_m
\end{bmatrix}
=
\begin{bmatrix}
\bS(z_1, z_1) & \bS(z_1, z_2) & \cdots & \bS(z_1, z_m)
\\
\bS(z_2, z_1) & \bS(z_2, z_2) & \cdots & \bS(z_2, z_m)
\\
\vdots & \ddots & \ddots & \vdots
\\
\bS(z_m, z_1) & \bS(z_m, z_2) & \cdots & \bS(z_m, z_m)
\end{bmatrix}^{-1}
 \begin{bmatrix}
w_1
\\
w_2\\
\vdots
\\
w_m
\end{bmatrix}.
\]
\end{theorem}
\begin{proof}
Consider the module map $X_{\clz, \clw}$ on the quotient module $\clq_{\clz}$ as (see Lemma \ref{eqn: Y star S})
\[
X_{\clz, \clw}^* \mathbb{S} (\cdot, z_j) = \bar{w}_j \mathbb{S} (\cdot, z_j),
\]
for all $j=1, \ldots, m$. Define
\begin{equation}\label{psi Z W}
\psi_{\clz, \clw} = X_{\clz, \clw}(P_{\clq_{\clz}} 1).
\end{equation}
We note the crucial fact that (as $\clq_\clz \subset H^\infty(\D^n)$, or see Corollary \ref{cor: lift without 1})
\[
\psi_{\clz, \clw} \in H^\infty(\D^n).
\]
\textsf{Claim:} A function $\vp \in \cls(\D^n)$ interpolates $\{z_i\}_{i=1}^m \subset \D^n$ and $\{w_i\}_{i=1}^m \subset \D$, that is
\[
\vp(z_i) = w_i,
\]
for all $i=1, \ldots, m$, if and only if
\[
S_{\vp} = X_{\clz, \clw}.
\]
Indeed, since $\mathbb{S}(\cdot, z_i) \in \clq_{\clz}$, it follows that $P_{\clq_{\clz}} \mathbb{S}(\cdot, z_i) = \mathbb{S}(\cdot, z_i)$ and hence
\[
\begin{split}
S_\vp^* \mathbb{S}(\cdot, z_i) & = P_{\clq_{\clz}} T_\vp^* \mathbb{S}(\cdot, z_i)
\\
& = \overline{\vp(z_i)} \mathbb{S}(\cdot, z_i),
\end{split}
\]
for all $i=1, \ldots, m$. The definition of $X_{\clz, \clw}$ now supports the claim. Of course, $\vp$ is a lift of $X_{\clz, \clw}$. Then, by Theorem \ref{thm: new CLT 1}, it follows that $\vp \in \cls(\D^n)$ interpolates $\{z_i\}_{i=1}^m$ and $\{w_i\}_{i=1}^m$ if and only if
\[
\mi_{\clz, \clw} f = \int_{\T^n} \psi_{\clz, \clw} f \,d \mu \qquad (f \in \clm_{\clq_\clz}),
\]
defines a contraction $\mi_{\clz, \clw}: (\clm_{\clq_\clz}, \|\cdot\|_1) \raro \mathbb{C}$. This proves the first half of the theorem. Now all that is left to do is calculate the representation of $\psi_{\clz, \clw}$. Corollary \ref{cor: lift without 1} says that
\[
X_{\clz, \clw} = S_{\vp} = S_{\psi_{\clz, \clw}}.
\]
Since $\psi_{\clz, \clw} \in \clq_\clz$, there exists scalars $\{c_i\}_{i=1}^m$ such that
\[
\psi_{\clz, \clw} = \sum_{i=1}^{m} c_i \bS(\cdot, z_i).
\]
To compute the coefficients $\{c_i\}_{i=1}^m$ of the preceding expansion, we employ both reproducing kernel Hilbert space methods and conventional linear algebra. Fix $j \in \{1, \ldots, m\}$. Then
\[
\begin{split}
X_{\clz, \clw}^* \bS(\cdot, z_j) & = S^*_{\psi_{\clz, \clw}} \bS(\cdot, z_j)
\\
& = \overline{\psi_{\clz, \clw}(z_j)} \bS(\cdot, z_j),
\end{split}
\]
where, on the other hand, $X_{\clz, \clw}^* \bS(\cdot, z_j) = \bar{w}_j \bS(\cdot, z_j)$. Therefore
\[
w_j = \psi_{\clz, \clw}(z_j),
\]
and hence, by the reproducing property of kernel functions \eqref{eqn: rep prop}, it follows that
\[
\begin{split}
w_j & = \psi_{\clz, \clw}(z_j)
\\
& = \Big\langle \psi_{\clz, \clw}, \bS(\cdot, z_j) \Big\rangle_{H^2(\T^n)}
\\
& = \Big\langle \sum_{i=1}^{m} c_i \bS(\cdot, z_i), \bS(\cdot, z_j) \Big\rangle_{H^2(\T^n)}
\\
& = \sum_{i=1}^{m} c_i \bS(z_j, z_i),
\end{split}
\]
for all $j=1, \ldots, m$. In other words, we have
\[
\begin{bmatrix}
\bS(z_1, z_1) & \bS(z_1, z_2) & \cdots & \bS(z_1, z_m)
\\
\bS(z_2, z_1) & \bS(z_2, z_2) & \cdots & \bS(z_2, z_m)
\\
\vdots & \ddots & \ddots & \vdots
\\
\bS(z_m, z_1) & \bS(z_m, z_2) & \cdots & \bS(z_m, z_m)
\end{bmatrix}
\begin{bmatrix}
c_1
\\
c_2\\
\vdots
\\
c_m
\end{bmatrix}
 =
 \begin{bmatrix}
w_1
\\
w_2\\
\vdots
\\
w_m
\end{bmatrix},
\]
equivalently
\[
\begin{bmatrix}
c_1
\\
c_2\\
\vdots
\\
c_m
\end{bmatrix}
=
\begin{bmatrix}
\bS(z_1, z_1) & \bS(z_1, z_2) & \cdots & \bS(z_1, z_m)
\\
\bS(z_2, z_1) & \bS(z_2, z_2) & \cdots & \bS(z_2, z_m)
\\
\vdots & \ddots & \ddots & \vdots
\\
\bS(z_m, z_1) & \bS(z_m, z_2) & \cdots & \bS(z_m, z_m)
\end{bmatrix}^{-1}
 \begin{bmatrix}
w_1
\\
w_2\\
\vdots
\\
w_m
\end{bmatrix}.
\]
Note that the $m \times m$ matrix
\[
\begin{bmatrix}
\bS(z_1, z_1) & \bS(z_1, z_2) & \cdots & \bS(z_1, z_m)
\\
\bS(z_2, z_1) & \bS(z_2, z_2) & \cdots & \bS(z_2, z_m)
\\
\vdots & \ddots & \ddots & \vdots
\\
\bS(z_m, z_1) & \bS(z_m, z_2) & \cdots & \bS(z_m, z_m)
\end{bmatrix},
\]
is nothing but the Gram matrix of the linearly independent kernel functions $\{\bS(\cdot, z_i): i =1, \ldots, m\}$. The invertibility of the matrix is now immediate.
\end{proof}

\begin{remark}
For solutions to the interpolation problem in the setting of bounded harmonic functions and $H^p$ functions, we refer the reader to Duren and Williams \cite{DW}. The setting of \cite{DW} is, in fact, that of more general Banach spaces, where the technique also utilizes a Hahn–Banach type extension theorem. The results in \cite{DW} follows the line of the interpolation problem originally settled by Carleson \cite{Car}.
\end{remark}

\section{Quantitative interpolation and examples}\label{sec: quant inter}

This section is a continuation of our investigation into the interpolation problem. To begin, we will provide a quantitative solution to the interpolation problem on $\D^n$. The quantitative solution will then be employed to generate examples of interpolation with interpolating functions in $\cls(\D^n)$, $n \geq 2$.

Let $\clz = \{z_i\}_{i=1}^m \subset \D^n$ be $m$ distinct points, and let $\{w_i\}_{i=1}^m \subset \D$ be $m$ scalars. As usual, define the $m$-dimensional zero-based quotient module $\clq_\clz$ of $H^2(\T^n)$ by
\begin{equation}\label{eqn: Q Z}
\clq_{\clz} = \text{span} \{\mathbb{S}(\cdot, z_i): i=1, \ldots, m\},
\end{equation}
and $X_{\clz, \clw} \in \clb(\clq_\clz)$ by
\[
X_{\clz, \clw}^* \mathbb{S}(\cdot, z_j) = \bar{w}_j \mathbb{S}(\cdot, z_j) \qquad (j=1, \ldots,m).
\]
As observed in Lemma \ref{eqn: Y star S}, $X_{\clz, \clw}$ is a module map, and hence Corollary \ref{cor: lift without 1} implies
\begin{equation}\label{eqn: S ZW = S psi}
X_{\clz, \clw} = S_\psi,
\end{equation}
where
\[
\psi := X_{\clz, \clw} (P_{\clq_\clz} 1).
\]
Recall that
\[
\psi \in \clq_\clz \subset H^\infty(\D^n).
\]
On the other hand, as observed in the proof of Theorem \ref{thm: interpolation} (more specifically, the claim part in the proof of Theorem \ref{thm: interpolation}), there exists a function $\vp \in \cls(\D^n)$ such that
\[
\vp(z_i) = w_i,
\]
for all $i=1, \ldots, m$, if and only if
\[
S_{\vp} = X_{\clz, \clw}.
\]
Equivalently, $X_{\clz, \clw}$ on $\clq_\clz$ is a contraction and admits a lift (namely, $\vp \in \cls(\D^n)$). Based on Theorem \ref{thm: new CLT 2} and the fact that $\psi = X_{\clz, \clw} (P_{\clq_\clz} 1)$, this is the same as saying that
\[
\text{dist}_{L^1(\T^n)}\Big(\frac{\bar \psi}{\|\psi\|_2^2}, \widetilde{\clm}_{\clq_\clz}\Big) \geq 1,
\]
where $\widetilde \clm_{\clq_\clz} = ({\clq}^{conj}_{\clz} \ominus \{\bar \psi\}) \dotplus (\clm_n \dotplus H^2_0(\T^n))$. This results in the quantitative solution to the interpolation problem:

\begin{theorem}\label{thm: quantitative inter}
Let $\clz = \{z_i\}_{i=1}^m \subset \D^n$ be $m$ distinct points, and let $\{w_i\}_{i=1}^m \subset \D$ be $m$ scalars. Define $\psi := X_{\clz, \clw} (P_{\clq_\clz} 1)$ and
\[
\widetilde \clm_{\clq_\clz} = ({\clq}^{conj}_{\clz} \ominus \{\bar \psi\}) \dotplus (\clm_n \dotplus H^2_0(\T^n)).
\]
Then there exists $\vp \in \cls(\D^n)$ such that
\[
\vp(z_i) = w_i,
\]
for all $i=1, \ldots, m$, if and only if
\[
\text{dist}_{L^1(\T^n)}\Big(\frac{\bar \psi}{\|\psi\|_2^2}, \widetilde{\clm}_{\clq_\clz}\Big) \geq 1.
\]
Moreover, in this case, we have
\[
\psi(z_i) = w_i,
\]
for all $i=1, \ldots, m$.
\end{theorem}

Here is how the proof of the final assertion works: For each $i=1, \ldots, m$, in view of the definition of $X_{\clz, \clw}$ and \eqref{eqn: S ZW = S psi}, we compute
\[
\begin{split}
\bar{w}_i \mathbb{S}(\cdot, z_i) & = X_{\clz, \clw}^* \mathbb{S}(\cdot, z_i)
\\
& = S_\psi^* \mathbb{S}(\cdot, z_i)
\\
& = P_{\clq_\clz} T_{\psi}^* \mathbb{S}(\cdot, z_i)
\\
& = \overline{\psi(z_i)} \mathbb{S}(\cdot, z_i),
\end{split}
\]
as $\mathbb{S}(\cdot, z_i) \in \clq_\clz$. Therefore, $\psi(z_i) = w_i$ for all $i=1, \ldots, m$, which completes the proof. The final assertion will play an important role in the discussion that follows.

The rest of this section will be devoted to exploring examples of interpolation. We need to prove two lemmas. Before doing so, let us standardize some notations. We will set aside $m \geq 2$ as the number of nodes of the given interpolation data. We use bold letters such as $\bm{a}$, $\bm{v}$, $\bm{w}$, etc. to denote vectors in $\mathbb{C}^m$. For instance
\[
\bm{a} = (a_1, \ldots, a_m) \in \mathbb{C}^m.
\]
Also, denote by $\langle \cdot, \cdot \rangle_{\mathbb{C}^m}$ the standard inner product on $\mathbb{C}^m$. In particular
\[
\|\bm{a}\|_{\mathbb{C}^m} = (\sum_{i=1}^{m}|a_i|^2)^{\frac{1}{2}}.
\]
We write
\[
\bm{a}^\perp = \{\bm{v} \in \mathbb{C}^m: \langle \bm{a}, \bm{v}\rangle_{\mathbb{C}^m} = 0\}.
\]
In view of the above notation, for each $\bm{a} \in {\mathbb{C}^m}$, we have the orthogonal decomposition
\[
\mathbb{C}^m = \mathbb{C} \bm{a} \oplus \bm{a}^\perp.
\]
We will work in the following general setting: Fix $m$ distinct points $\clz = \{z_i\}_{i=1}^m \subset \D^n$ and $m$ scalars $\{w_i\}_{i=1}^m \subset \D$. The quotient module of interest will be $\clq_\clz \subset H^\infty(\D^n)$ as defined in \eqref{eqn: Q Z}.

\begin{lemma}\label{lemma: repr of psi}
Let $\psi \in \clq_\clz$, and suppose $\psi(z_i) = w_i$ for all $i=1, \ldots, m$. Then there exit $\bm{v} \in \bm{w}^\perp$ such that
\[
\psi = \frac{\|\psi\|^2_2}{\|\bm{w}\|_{\mathbb{C}^m}^2} \sum_{i=1}^{m} w_i \mathbb{S}(\cdot, z_i) + \sum_{i=1}^{m} v_i \mathbb{S}(\cdot, z_i).
\]
\end{lemma}
\begin{proof}
Since $\psi \in \clq_\clz$, there exists $\bm{c} \in \mathbb{C}^m$ such that
\[
\psi = \sum_{i=1}^{m} c_i \mathbb{S}(\cdot, z_i).
\]
Moreover, there exist a scalar $\alpha \in \mathbb{C}$ and a vector $\bm{v} \in \bm{c}^\perp$ such that $\bm{c} = \alpha \bm{w} \oplus \bm{v}$. Then
\[
\psi = \alpha \sum_{i=1}^{m} w_i \mathbb{S}(\cdot, z_i) + \sum_{i=1}^{m} v_i \mathbb{S}(\cdot, z_i).
\]
By assumption, $\psi \in H^\infty(\D^n)$ and $\psi(z_i) = w_i$ for all $i=1, \ldots, m$. The above equality then results in
\[
\begin{split}
\|\psi\|_2^2 & = \Big\langle \alpha \sum_{i=1}^{m} w_i \mathbb{S}(\cdot, z_i) + \sum_{i=1}^{m} v_i \mathbb{S}(\cdot, z_i), \psi \Big\rangle_{H^2(\T^n)}
\\
& = \alpha \Big\langle T_{\psi}^*\Big(\sum_{i=1}^{m} w_i \mathbb{S}(\cdot, z_i) + \sum_{i=1}^{m} v_i \mathbb{S}(\cdot, z_i)\Big), 1 \Big\rangle_{H^2(\T^n)}
\\
& = \alpha \Big\langle \Big(\sum_{i=1}^{m} w_i \overline{\psi(z_i)} \mathbb{S}(\cdot, z_i) + \sum_{i=1}^{m} v_i  \overline{\psi(z_i)}\mathbb{S}(\cdot, z_i)\Big), 1 \Big\rangle_{H^2(\T^n)}
\\
& = \alpha \Big\langle \Big(\sum_{i=1}^{m} |w_i|^2 \mathbb{S}(\cdot, z_i) + \sum_{i=1}^{m} v_i \bar{w}_i \mathbb{S}(\cdot, z_i)\Big), 1 \Big\rangle_{H^2(\T^n)}
\\
& = \alpha \|\bm{w}\|^2_{\mathbb{C}^m} + \sum_{i=1}^{m} v_i \bar{w}_i
\\
& =  \alpha \|\bm{w}\|^2_{\mathbb{C}^m},
\end{split}
\]
as $\bm{v} \perp \bm{w}$. We have also used the general property that $\langle \mathbb{S}(\cdot, w), 1 \rangle_{H^2(\T^n)} = 1$ for all $w \in \D^n$. The above identity yields
\[
\alpha = \frac{\|\psi\|_2^2}{\|\bm{w}\|^2_{\mathbb{C}^m}},
\]
which completes the proof of the lemma.
\end{proof}

The proof of the following lemma is similar to the proof of the previous one.

\begin{lemma}\label{lemma: QZ minus psi}
Let $\psi \in \clq_\clz$, and suppose $\psi(z_i) = w_i$ for all $i=1, \ldots, m$. Then
\[
\clq_\clz \ominus \{\psi\} = \Big\{ \sum_{i=1}^{m} v_i \mathbb{S}(\cdot, z_i): \bm{v} \in \bm{w}^\perp\Big\}.
\]
\end{lemma}
\begin{proof}
Given $\bm{v} \in \mathbb{C}^m$, observe that
\[
\sum_{i=1}^{m} v_i \mathbb{S}(\cdot, z_i) \perp \psi,
\]
if and only if
\[
\begin{split}
0 & = \Big\langle \sum_{i=1}^{m} v_i \mathbb{S}(\cdot, z_i), \psi \Big\rangle_{H^2(\T^n)}
\\
& = \Big\langle T^*_{\psi} \Big(\sum_{i=1}^{m} v_i \mathbb{S}(\cdot, z_i)\Big), 1 \Big\rangle_{H^2(\T^n)}
\\
& = \Big\langle\sum_{i=1}^{m} v_i \overline{\psi(z_i)} \mathbb{S}(\cdot, z_i), 1 \Big\rangle_{H^2(\T^n)}
\\
& = \sum_{i=1}^{m} v_i \bar{w}_i
\\
& = \langle \bm{v}, \bm{w} \rangle_{\mathbb{C}^m}.
\end{split}
\]
This completes the proof of the lemma.
\end{proof}

Now we are ready for examples of interpolation on $\D^n$, $n \geq 2$. First, we elaborate on the construction of the $3$-point interpolation problem. Consider the following setting:

\begin{enumerate}
\item $\{\bm{b}_0, \bm{b}_1, \bm{b}_2\}$ is an orthogonal basis for $\mathbb{C}^3$, where
\[
\begin{cases}
\bm{b}_0 = (1, 1, 1)
\\
\bm{b}_1 = (\zeta_{11}, \zeta_{12}, \zeta_{13})
\\
\bm{b}_2 = (\zeta_{21}, \zeta_{22}, \zeta_{23}).
\end{cases}
\]

\item $\|\bm{b}_1\|_{\mathbb{C}^3}, \|\bm{b}_2\|_{\mathbb{C}^3} \geq 1$.
\item $\{z_1, z_2, z_3\} \subset \D^n$ are three distinct points such that
\[
\begin{cases}
z_1 = (\zeta_{11}, \zeta_{21}, \tilde{z}_1)
\\
z_2 = (\zeta_{12}, \zeta_{22}, \tilde{z}_2)
\\
z_3 = (\zeta_{13}, \zeta_{23}, \tilde{z}_3),
\end{cases}
\]
for some (arbitrary) $\tilde{z}_1, \tilde{z}_1, \tilde{z}_3 \in \D^{n-2}$.
\item $\bm{w} = (w_1, w_2, w_3) \in \D^3$ such that $\|\bm{w}\|_{\mathbb{C}^3} \leq \frac{1}{\sqrt 3}$.
\end{enumerate}

\noindent \textsf{Claim:} There exists $\vp \in \cls(\D^n)$ such that $\vp(z_i) = w_i$ for all $i=1, 2,3$.

\noindent Here is how the proof of the claim goes: In view of Theorem \ref{thm: quantitative inter}, it is enough to prove that
\[
\text{dist}_{L^1(\T^n)}\Big(\frac{\bar \psi}{\|\psi\|_2^2}, \widetilde{\clm}_{\clq_\clz}\Big) \geq 1,
\]
where $\psi = X_{\clz, \clw} (P_{\clq_\clz} 1)$ and
\[
\widetilde \clm_{\clq_\clz} = ({\clq}^{conj}_{\clz} \ominus \{\bar \psi\}) \dotplus (\clm_n \dotplus H^2_0(\T^n)).
\]
Recall that $X_{\clz, \clw} \in \clb(\clq_\clz)$ is defined by
\[
X_{\clz, \clw}^* \mathbb{S}(\cdot, z_i) = \bar{w}_i \mathbb{S}(\cdot, z_i),
\]
for all $i=1, \ldots, m$. Also recall the crucial fact that (see  Theorem \ref{thm: quantitative inter})
\[
\psi(z_i) = w_i \qquad (i=1, \ldots, m).
\]
Using the conjugation invariance property of $L^1$-norm (that is, $\|f\|_{L^1(\T^n)} = \|\bar f\|_{L^1(\T^n)}$ for all $f \in L^1(\T^n)$), we infer that
\[
\text{dist}_{L^1(\T^n)}\Big(\frac{\bar \psi}{\|\psi\|_2^2}, \widetilde{\clm}_{\clq_\clz}\Big) = \text{dist}_{L^1(\T^n)}\Big(\frac{\psi}{\|\psi\|_2^2}, {\widetilde{\clm}_{\clq_\clz}}^{conj}\Big),
\]
where (recall that $\clm_n^{conj} = \clm_n$)
\[
{\widetilde{\clm}_{\clq_\clz}}^{conj} = (\clq_{\clz} \ominus \{\psi\}) \dotplus (\clm_n \dotplus {H^2_0(\T^n)}^{conj}).
\]
It will be convenient (as well as enough) to prove that
\[
\text{dist}_{L^1(\T^n)}\Big(\frac{\psi}{\|\psi\|_2^2}, {\widetilde{\clm}_{\clq_\clz}}^{conj}\Big) \geq 1.
\]
Also, to avoid notational confusion, we use $\{Z_1, \ldots, Z_n\}$ for the variables of $\mathbb{C}^n$. By the definition of Szeg\"{o} kernel, we have
\begin{equation}\label{eqn: series of kernel}
\begin{cases}
\mathbb{S}(\cdot, z_1) = 1 + \bar{\zeta}_{11} Z_1 + \bar{\zeta}_{21} Z_2 + \cdots
\\
\mathbb{S}(\cdot, z_2) =  1 + \bar{\zeta}_{12} Z_1 + \bar{\zeta}_{22} Z_2 + \cdots
\\
\mathbb{S}(\cdot, z_3) =  1 + \bar{\zeta}_{13} Z_1 + \bar{\zeta}_{23} Z_2 + \cdots .
\end{cases}
\end{equation}
We will need to prove the following inequality
\[
\Big\|\frac{\psi}{\|\psi\|_2^2} + f\Big\|_{L^1(\T^n)} \geq 1 \qquad (f \in {\widetilde{\clm}_{\clq_\clz}}^{conj}).
\]
Since $\psi(z_i) = w_i$ for all $i=1, \ldots, m$, in view of Lemma \ref{lemma: QZ minus psi}, an element $f \in {\widetilde{\clm}_{\clq_\clz}}^{conj}$ admits the following representation
\[
f = \sum_{i=1}^{3} v_i \mathbb{S}(\cdot, z_i) + \tilde f,
\]
for some $\bm{v} \in \bm{w}^\perp$ and $\tilde f \in \clm_n \dotplus {H^2_0(\T^n)}^{conj}$. Therefore, for each $f \in {\widetilde{\clm}_{\clq_\clz}}^{conj}$, by Lemma \ref{lemma: repr of psi}, we conclude that
\[
\frac{\psi}{\|\psi\|_2^2} + f = \sum_{i=1}^{3} \frac{w_i}{\|\bm{w}\|_{\mathbb{C}^3}^2} \mathbb{S}(\cdot, z_i) + \sum_{i=1}^{3} v_i \mathbb{S}(\cdot, z_i) + \tilde f,
\]
for some $\bm{v} \in \bm{w}^\perp$ and $\tilde f \in \clm_n \dotplus {H^2_0(\T^n)}^{conj}$. For each $\bm{v} \in \bm{w}^\perp$, we set
\[
F_{\bm{v}} =  \frac{1}{\|\bm{w}\|_{\mathbb{C}^3}^2} \bm{w} \oplus \bm{v}.
\]
It is important to keep in mind that $\bm{v}$ and $\tilde f$ depend on $f$. By assumption, $\|\bm{w}\|_{\mathbb{C}^3} \leq \frac{1}{\sqrt 3}$, and hence
\[
\|F_{\bm{v}}\|_{\mathbb{C}^3} \geq \sqrt 3.
\]
Using the kernel functions' power series expansion as in \eqref{eqn: series of kernel}, we find
\[
\begin{split}
\frac{\psi}{\|\psi\|_2^2} + f & = \sum_{i=1}^{3} \frac{w_i}{\|\bm{w}\|_{\mathbb{C}^3}^2} \mathbb{S}(\cdot, z_i) + \sum_{i=1}^{3} v_i \mathbb{S}(\cdot, z_i) + \tilde f
\\
& = \Big(\langle F_{\bm{v}}, \bm{b}_0 \rangle_{\mathbb{C}^3} 1 + \langle F_{\bm{v}}, \bm{b}_1 \rangle_{\mathbb{C}^3} Z_1 + \langle F_{\bm{v}}, \bm{b}_2 \rangle_{\mathbb{C}^3} Z_2 + \cdots \Big) + \tilde f.
\end{split}
\]
There exists $i \in \{0,1,2\}$ such that
\[
| \langle F_{\bm{v}}, \bm{b}_i \rangle_{\mathbb{C}^3}| \geq 1.
\]
If not, suppose $| \langle F_{\bm{v}}, \bm{b}_i \rangle_{\mathbb{C}^m}| < 1$ for all $i=0, 1,2$. Then
\[
\Big|\Big\langle F_{\bm{v}}, \|\bm{b}_i\|_{\mathbb{C}^m} \Big(\frac{1}{\|\bm{b}_i\|_{\mathbb{C}^3}} \bm{b}_i \Big) \Big\rangle\Big| < 1,
\]
implies
\[
\Big|\Big\langle F_{\bm{v}}, \Big(\frac{1}{\|\bm{b}_i\|_{\mathbb{C}^3}} \bm{b}_i \Big) \Big\rangle\Big| < \frac{1}{\|\bm{b}_i\|_{\mathbb{C}^3}} \leq 1,
\]
for all $i=0, 1,2$. Since
\[
\Big\{\frac{1}{\|\bm{b}_i\|_{\mathbb{C}^3}} \bm{b}_i\Big\}_{i=0}^2,
\]
is an orthonormal basis for $\mathbb{C}^3$, the above inequality contradicts the fact that $\|F_{\bm{v}}\|_{\mathbb{C}^3} \geq \sqrt 3$. On the other hand, since $\tilde f$ does not have an analytic part and
\[
\langle \tilde f, 1 \rangle_{L^2(\T^n)} = 0,
\]
it follows that
\[
\Big\langle \frac{\psi}{\|\psi\|_2^2} + f, g \Big\rangle_{L^2(\T^n)} =
\begin{cases}
\langle  F_{\bm{v}}, \bm{b}_0 \rangle_{\mathbb{C}^3} & \mbox{if } g = 1 \\
\langle F_{\bm{v}}, \bm{b}_1 \rangle_{\mathbb{C}^3} & \mbox{if } g = Z_1 \\
\langle F_{\bm{v}}, \bm{b}_2 \rangle_{\mathbb{C}^3} & \mbox{if } g = Z_2.
\end{cases}
\]
Therefore
\[
\Big|\Big\langle \frac{\psi}{\|\psi\|_2^2} + f, g \Big\rangle_{L^2(\T^n)}\Big| \geq 1,
\]
for some $g \in \{1, Z_1, Z_2\}$. On the other hand (see the duality \eqref{eqn: duality})
\[
\Big\langle \frac{\psi}{\|\psi\|_2^2} + f, g \Big\rangle_{L^2(\T^n)} =
\begin{cases}
\chi_1\Big(\frac{\psi}{\|\psi\|_2^2} + f\Big) & \mbox{if } g = 1 \\
\chi_{\bar{Z}_1} \Big(\frac{\psi}{\|\psi\|_2^2} + f\Big) & \mbox{if } g = Z_1 \\
\chi_{\bar{Z}_2}\Big(\frac{\psi}{\|\psi\|_2^2} + f\Big) & \mbox{if } g = Z_2.
\end{cases}
\]
However
\[
\|\chi_{\bar{g}}\| = 1,
\]
for all $g \in \{1, {Z}_1, {Z}_2\}$, and hence
\[
\Big\|\frac{\psi}{\|\psi\|_2^2} + f\Big\|_{L^1(\T^n)} \geq 1,
\]
for all $f \in \widetilde{\clm}_{\clq_\clz}^{conj}$. This completes the proof of the claim. Furthermore, in this case, we can specify an explicit interpolating function. Note that $\{\bm{e}_i\}_{i=0}^2$ is an orthonormal basis for $\mathbb{C}^3$, where
\[
\bm{e}_i = \frac{1}{\|\bm{b}_i\|_{\mathbb{C}^3}}\bm{b}_i,
\]
for all $i=0,1,2$. We write
\[
\bm{w} = \sum_{i=0}^2 \alpha_i \bm{e}_i,
\]
and set
\[
\vp(Z) = \frac{\alpha_0}{\|\bm{b}_0\|_{\mathbb{C}^3}} + \frac{\alpha_1}{\|\bm{b}_1\|_{\mathbb{C}^3}} Z_1 +  \frac{\alpha_2}{\|\bm{b}_2\|_{\mathbb{C}^3}} Z_2,
\]
for all $Z = (Z_1, \ldots, Z_n) \in \D^n$. Since $\|\bm{w}\|_{\mathbb{C}^3}^2 \leq \frac{1}{3}$, it follows that
\[
\sum_{i=0}^2 |\alpha_i|^2 \leq \frac{1}{3},
\]
and hence, by the Cauchy-Schwarz inequality, we conclude that
\[
\sum_{i=0}^2 |\alpha_i| \leq 1.
\]
Moreover, since $\|\bm{b}_i\|_{\mathbb{C}^3} \geq 1$ for all $i=0,1,2$, for each $Z \in \D^n$, we infer that
\[
\begin{split}
|\vp(Z)| & = \Big|\frac{\alpha_0}{\|\bm{b}_0\|_{\mathbb{C}^3}} + \frac{\alpha_1}{\|\bm{b}_1\|_{\mathbb{C}^3}} Z_1 +  \frac{\alpha_2}{\|\bm{b}_2\|_{\mathbb{C}^3}} Z_2 \Big|
\\
& \leq \frac{|\alpha_0|}{\|\bm{b}_0\|_{\mathbb{C}^3}} + \frac{|\alpha_1|}{\|\bm{b}_1\|_{\mathbb{C}^3}} |Z_1| +  \frac{|\alpha_2|}{\|\bm{b}_2\|_{\mathbb{C}^3}} |Z_2|
\\
& \leq \sum_{i=0}^2 |\alpha_i|
\\
& \leq 1,
\end{split}
\]
and consequently, $\vp \in \cls(\D^n)$. Finally, we compute
\[
\begin{split}
\vp(z_i) & = \frac{\alpha_0}{\|\bm{b}_0\|_{\mathbb{C}^3}} + \frac{\alpha_1}{\|\bm{b}_1\|_{\mathbb{C}^3}} \zeta_{1i} +  \frac{\alpha_2}{\|\bm{b}_2\|_{\mathbb{C}^3}} \zeta_{2i}
\\
& = \alpha_0 \pi_i(\bm{e}_0) + \alpha_1 \pi_i(\bm{e}_1) + \alpha_2 \pi_i(\bm{e}_2)
\\
& = \pi_i (\sum_{j=0}^2 \alpha_j \bm{e}_j)
\\
& = \pi_i(\bm{w})
\\
& = w_i,
\end{split}
\]
for all $i=1,2,3$. Therefore, $\vp$ is a solution to the interpolation problem with data $\{z_i\}_{i=1}^3 \subset \D^n$ and $\{w_i\}_{i=1}^3 \subset \D$.

For general $m$-point interpolation, $m \geq 2$, the same proof concept applies, but the computation would be more laborious. We only report the general result and leave the other details to the interested readers.

\begin{theorem}\label{thm: eg interpolation}
Let $n \geq 2$, $m \geq 3$, and suppose $n \geq m-1$. Let $\{z_i\}_{i=1}^m \subset \D^n$ be $m$ distinct points, $\{w_i\}_{i=1}^m \subset \D$ be $m$ scalars, and let $\{\bm{b}_i\}_{i=0}^{m-1} \subset \mathbb{C}^m$, where $\bm{b}_0 = (1, \ldots, 1)$, and
\[
\bm{b}_j= (\zeta_{j1}, \zeta_{j2}, \ldots, \zeta_{jm}),
\]
for all $j=1, \ldots, m-1$. Assume that:
\begin{enumerate}
\item $\{\bm{b}_i\}_{i=0}^{m-1}$ is an orthogonal basis for $\mathbb{C}^m$.
\item $\|\bm{b}_i\| _{\mathbb{C}^m} \geq 1$ for all $i=1, \ldots, m-1$.
\item $z_j = (\zeta_{1j}, \zeta_{2j}, \ldots, \zeta_{m-1,j}, \tilde{z}_j)$, where $\tilde{z}_j \in \D^{n-m+1}$ arbitrary, and $j=1, \ldots,m$.
\item $\|\bm{w}\|_{\mathbb{C}^m} \leq \frac{1}{\sqrt n}$, where $\bm{w} = (w_1, \ldots, w_m)$.
\end{enumerate}

Then there exists $\vp \in \cls(\D^n)$ such that
\[
\vp(z_i) = w_i,
\]
for all $i=1, \ldots, m$. Furthermore, $\vp$ can be chosen as a polynomial.
\end{theorem}

Evidently, there is no dearth of examples of data that meet the aforementioned conditions. The following remark elaborates on this:

\begin{remark}
If the number of variables $n (\geq 2)$ and the number of nodes $m (\geq 3)$ satisfies the condition $n \geq m-1$, and if one restricts the first $m-1$ slots of the coordinates of the interpolation nodes $\{z_i\}_{i=1}^m$ (so that the corresponding columns along with the constant vector $1$ forms a basis of $\mathbb{C}^m$) along with the norm bound on $\bm{w}$ as
\[
\|\bm{w}\|_{\mathbb{C}^m} \leq \frac{1}{\sqrt{n}},
\]
then one can ensure that interpolation will occur for any choice of $\{\tilde{z}_i\}_{i=1}^m \subset \D^{n-m+1}$. The relationship between the orthogonal set of vectors $\{\bm{b}_i\}_{i=1}^{m-1} \subset \mathbb{C}^m$ and interpolation nodes $\{z_i\}_{i=1}^m \subset \D^n$ can be represented by the formal matrix:
\[
\bordermatrix{
&\bm{b}_1 & \bm{b}_2 & \bm{b}_3 & \cdots & \bm{b}_{m-1} \cr
z_1 & \zeta_{11} & \zeta_{21} & \zeta_{31} & \cdots & \zeta_{m-1,1} & \cdots \cr
z_2 &\zeta_{12} & \zeta_{22} & \zeta_{32} & \cdots & \zeta_{m-1,2} & \cdots \cr
\vdots & \vdots & \vdots & \vdots & \vdots & \vdots & \ddots \cr
z_m & \zeta_{1m} & \zeta_{2m} & \zeta_{3m} & \cdots & \zeta_{m-1,m} & \cdots \cr
}.
\]
What this means is that there is an abundance of examples of interpolation in hand in several variables.
\end{remark}

We refer the reader to \cite{MP} for interpolation from operator algebraic perspective.

\section{Commutant lifting and examples}\label{subsect: example verification}

This section contains illustrations of commutant lifting on quotient modules of $H^2(\T^n)$, $n > 1$. Our first aim is to validate the examples in Section \ref{sec: examples of hom qm} using our commutant lifting theorem. We begin with a lemma.

\begin{lemma}\label{lemma: h 1 and 2 norm}
Let $h \in H^2(\T^n)$. Then $\|h\|_1 = \|h\|_2 = 1$ if and only if $h$ is inner.
\end{lemma}
\begin{proof}
Suppose $\|h\|_1 = \|h\|_2 = 1$. In particular, $h \in H^1(\T^n) \subseteq L^1(\T^n)$. By the Hahn–Banach theorem, there exists $\vp \in L^\infty(\T^n)$ such that $\|\vp\|_\infty = 1$ (as $\|h\|_1 = 1$) and
\[
\int_{\T^n} h \vp d\mu = \|h\|_1 = 1.
\]
In the above, we used the duality $(L^1(\T^n))^* \cong L^\infty(\T^n)$ once more. We claim that $\vp$ is unimodular. Indeed, if
\[
|\vp| < 1 \text{ on } A,
\]
for some measurable set $A \subseteq \T^n$ such that $\mu(A) > 0$, then
\[
\begin{split}
1 & = \Big|\int_{\T^n} h \vp d\mu\Big|
\\
& \leq \Big|\int_{A^c} h \vp d\mu\Big| + \Big|\int_{A} h \vp d\mu\Big|
\\
& \leq \int_{A^c} |h| |\vp| d\mu + \int_{A} |h| |\vp| d\mu
\\
& < \int_{A^c} |h| d\mu + \int_{A} |h| d\mu
\\
& = \|h\|_1,
\end{split}
\]
that is, $1 < \|h\|_1$, a contradiction. Since $h \in H^2(\D^n) \subseteq L^2(\T^n)$, we find a scalar $c$ and a function $g \in L^2(\T^n)$ such that
\[
\vp = c \bar h \oplus g.
\]
Observe that $\langle \bar h, g \rangle = \langle h, \bar g \rangle = 0$. Therefore
\[
\begin{split}
1 & = \int_{\T^n} h \vp d\mu
\\
& = \int_{\T^n} h (c \bar h \oplus g) d\mu
\\
& = \langle h, \bar c h \oplus \bar g \rangle_{L^2(\T^n)}
\\
& = c,
\end{split}
\]
and hence, $\vp = \bar h \oplus g$. Then
\[
\begin{split}
1 + \|g\|_2^2 & = \|h\|_2^2 + \|g\|_2^2
\\
& = \|\vp\|_2^2
\\
& \leq \|\vp\|_\infty^2
\\
& =1,
\end{split}
\]
implies that $g = 0$, and hence $\vp = \bar h$. Since $\vp \in L^\infty(\T^n)$, it follows that $h \in H^\infty(\D^n)$ is an inner function. The converse simply follows from the integral representation of norms on $H^2(\T^n)$ and $H^1(\T^n)$ and the fact that $|h| = 1$ a.e. on $\T^n$.
\end{proof}

Now we follow the setting of Corollary \ref{cor: hom poly fail}: For a fixed $m \in \mathbb{N}$, we consider the homogeneous quotient module
\[
\clq_m = \bigoplus_{t=0}^m H_t,
\]
a homogeneous polynomial $p \in \clq_m$ as
\[
p = \sum_{|k|=m} a_k z^k,
\]
with $\|p\|_2 = 1$, and that $a_k, a_l \neq 0$ for some $k \neq l$ in $\Z_+^n$. We know, by Theorem \ref{thm: new CLT 1}, that $S_p$ is liftable if and only if
\[
X_{\clq_m} (f) = \int_{\T^n} \psi f d\mu \qquad (f \in \clm_{\clq_m}),
\]
defines a contraction on $(\clm_{\clq_m}, \|\cdot\|_1)$, where $\psi = S_p(P_{\clq_m} 1)$. Since $1$ and $p$ are in $\clq_m$, it follows that $\psi = p$, and hence
\[
\begin{split}
X_{\clq_m} (\bar{p}) & = \int_{\T^n} \bar{p} \psi d\mu
\\
& = \int_{\T^n} |p|^2 d\mu
\\
& = 1.
\end{split}
\]
However
\[
\|\bar p\|_1 = \|p\|_1 < 1.
\]
Indeed, since $a_k, a_l \neq 0$, Lemma \ref{lem: inner poly} ensures that $p$ is not inner. This, together with the fact that $\|p\|_2 = 1$ and Lemma \ref{lemma: h 1 and 2 norm} completes the proof of the claim. Therefore, $X_{\clq_m}$ on $(\clm_{\clq_m}, \|\cdot\|_1)$ is not a contraction, and hence $S_p$ is not liftable. As a result, we recover Corollary \ref{cor: hom poly fail} using Theorem \ref{thm: new CLT 1}.

The idea used in the preceding example can be extended to provide further nontrivial examples of module maps that do not admit any lift. The following is an example, and this time we will use Theorem \ref{thm: new CLT 1} directly to prove that such a module map does not lift. Let $n > 1$. Consider the submodule
\[
\cls = z_1 \cdots z_n H^2(\T^n).
\]
We will be working on the corresponding quotient module $\clq = \cls^\perp$. Clearly
\[
\clq = \ker \Big(\prod_{i=1}^{n} T_{z_i}^*\Big).
\]
We observe that
\[
\clq = H^2_{{z}_1}(\T^n) + \cdots + H^2_{{z}_n}(\T^n),
\]
where $H^2_{{z}_i}(\T^n)$, $i=1, \ldots, n$, is the closed subspace of $H^2(\T^n)$ of functions that are independent of the $z_i$ variable, or equivalently
\[
H^2_{z_i}(\T^n) = \ker T_{z_i}^*.
\]
Indeed, it is clear that $H^2_{{z}_1}(\T^n) + \cdots + H^2_{{z}_n}(\T^n) \subseteq \clq$. Let $f \in \ker (\prod_{i=1}^{n} T_{z_i}^*)$, and suppose
\[
f = \sum_{k \in \Z_+^n} a_k z^k.
\]
Then
\[
\begin{split}
0 &= T_{z_1 \cdots z_n}^* f
\\
& = P_{H^2(\T^n)} (\sum_{k \in \Z_+^n} a_k \bar{z}_1 \cdots \bar{z}_n z^k),
\end{split}
\]
implies
\[
\sum_{k \in \Z_+^n} a_k \bar{z}_1 \cdots \bar{z}_n z^k \in (H^2(\T^n))^\perp.
\]
In other words, if $a_k \neq 0$ for some $k = (k_1, \ldots, k_n) \in \Z_+^n$, then we must have $k_i = 0$ for some $i=1, \ldots, n$. Therefore, there exists $f_i \in H^2_{z_i}(\T^n)$, $i=1, \ldots, n$, such that $f = f_1 + \cdots + f_n$. This proves the claim. Now for each $i=1, \ldots, n$, set
\[
\zeta_i:= \prod_{j\neq i} z_j,
\]
and pick inner function $\vp_i \in H^2_{z_i}(\T^n)$. Let $z_0 \in \T^n$ and suppose $\vp_i(z_0)$ is well-defined and
\[
|\vp_i(z_0)| = 1,
\]
for all $i=1, \ldots, n$. Choose $\{\alpha_1, \ldots, \alpha_n\} \subset \mathbb{R}_{\geq 0}$ such that $\alpha_p, \alpha_q \neq 0$ for some $p\neq q$, and
\[
\sum_{i=1}^{n} \alpha_i^2 = 1.
\]
The preceding set of assumptions ensures that
\[
\sum_{i=1}^{n} \alpha_i >1.
\]
Finally, define $\vp \in \cls(\D^n)$ by
\[
\vp = \sum_{i=1}^{n} \alpha_i \bar{\beta}_i \zeta_i \vp_i,
\]
where $\beta_i = (\zeta_i \vp_i)(z_0)$ for all $i=1, \ldots, n$. We claim that $\vp$ is not inner. Indeed, since
\[
\vp(z_0) = \sum_{i=1}^{n} \alpha_i \bar{\beta}_i \beta_i,
\]
and $|\beta_i| = 1$, it follows that
\[
\vp(z_0) = \sum_{i=1}^{n} \alpha_i > 1.
\]
Therefore, there exists $r \in (0,1)$ such that (note that $\vp$ is well-defined at $z_0$)
\[
|\vp(rz_0)| > 1,
\]
and hence, by the maximum modulus theorem, we conclude that $\|\vp\|_\infty > 1$, which completes the proof of the claim. Next, we claim that $S_\vp$ is a contraction. Fix $f \in \clq$. For each $i \in \{1, \ldots, n\}$, we have
\[
\clq \ominus H^2_{z_i}(\T^n) = (\ker T_{z_i}^*)^\perp \cap \clq = \{z_i g \in \clq: g \in H^2(\T^n)\},
\]
and hence there exist $f_i \in H^2_{z_i}(\T^n)$ and $g_i \in H^2(\T^n)$ such that
\[
f = f_i \oplus z_i g_i \in H^2_{z_i}(\T^n) \oplus (\clq \ominus H^2_{z_i}(\T^n)).
\]
Then
\[
\begin{split}
S_{\zeta_i \vp_i} f & = S_{\zeta_i \vp_i} (f_i + z_i g_i)
\\
& = S_{\zeta_i \vp_i} f_i + P_{\clq}(\zeta_i \vp_i z_i g_i)
\\
& = S_{\zeta_i \vp_i} f_i + P_{\clq}(z_1 \cdots z_n \vp_i g_i)
\\
& = S_{\zeta_i \vp_i} f_i,
\end{split}
\]
as $z_1 \cdots z_n \vp_i g_i \in \cls$, and hence
\[
S_{\zeta_i \vp_i} f = S_{\zeta_i \vp_i} P_{H^2_{z_i}(\T^n)} f.
\]
Observe moreover that $\zeta_i \vp_i H^2_{z_i}(\T^n) \subseteq H^2_{z_i}(\T^n)$. In view of $H^2_{z_i}(\T^n) \subseteq \clq$, we conclude that $S_{\zeta_i \vp_i} P_{H^2_{z_i}(\T^n)} f = \zeta_i \vp_i P_{H^2_{z_i}(\T^n)} f$, which yields
\[
S_{\zeta_i \vp_i} f = \zeta_i \vp_i P_{H^2_{z_i}(\T^n)} f.
\]
Therefore, for $i \neq j$, we have
\[
\begin{split}
\langle S_{\zeta_i \vp_i} f, S_{\zeta_j \vp_j} f \rangle & = \langle \zeta_i \vp_i P_{H^2_{z_i}(\T^n)} f, \zeta_j \vp_j P_{H^2_{z_j}(\T^n)} f \rangle
\\
& = \langle T_{\zeta_j}^* \zeta_i \vp_i P_{H^2_{z_i}(\T^n)} f,  \vp_j P_{H^2_{z_j}(\T^n)} f \rangle
\\
& = \langle T_{z_i}^* z_j \vp_i P_{H^2_{z_i}(\T^n)} f,  \vp_j P_{H^2_{z_j}(\T^n)} f \rangle
\\
& = 0,
\end{split}
\]
as $z_j \vp_i P_{H^2_{z_i}(\T^n)} f \in \ker T_{z_i}^*$. So we find
\begin{equation}\label{eqn: S zeta phi ortho}
S_{\zeta_i \vp_i} f \perp S_{\zeta_j \vp_j} f \qquad (i \neq j).
\end{equation}
This allows us to compute the norm of $S_\vp f$ as follows (note that $|\beta_i|=1$ and $S_{\zeta_i \vp_i}$ is a contraction for all $i=1, \ldots, n$):
\[
\begin{split}
\|S_\vp f\|^2 & = \|\sum_{i=1}^{n} \alpha_i \bar{\beta}_i S_{\zeta_i \vp_i} f\|^2
\\
& = \sum_{i=1}^{n} \alpha_i^2 \|S_{\zeta_i \vp_i} f\|^2
\\
& \leq \sum_{i=1}^{n} \alpha_i^2 \|f\|^2
\\
& = \|f\|^2.
\end{split}
\]
This means that $S_\vp$ is a contraction. Our final claim is that $S_\vp$ is incapable of admitting any lift, which, in view of Theorem \ref{thm: new CLT 1}, is equivalent to the assertion that $X_{\clq} : (\clm_\clq, \|\cdot\|_1) \raro \mathbb{C}$ is not a contraction, where
\[
X_{\clq} f = \int_{\T^n} \psi f d\mu \qquad (f \in \clm_{\clq}),
\]
and
\[
\psi = S_\vp (P_{\clq} 1).
\]
Indeed, since $1, \vp \in \clq$, it follows that
\[
\psi = \vp.
\]
On the other hand, since $\bar{\vp} \in \clm_\clq$ (recall that $\clm_\clq = \clq^{conj} \dotplus (\clm_n \dotplus H^2_0(\T^n))$), we observe that
\[
\begin{split}
X_{\clq} \bar{\vp} & = \int_{\T^n} \vp \bar{\vp} d\mu
\\
& = \|\vp\|_{H^2(\T^n)}^2
\\
& = 1.
\end{split}
\]
Finally, applying \eqref{eqn: S zeta phi ortho} to $f = 1 \in \clq$, we obtain that
\[
\|\vp\|_{H^2(\T^n)}^2 = 1.
\]
This also follows from the equalities following  \eqref{eqn: S zeta phi ortho} corresponding to the choice $f=1$ along with the fact that $\zeta_i \vp_i$ is inner for all $i=1, \ldots, n$. Since $\vp$ is not an inner function, by Lemma \ref{lemma: h 1 and 2 norm}, we conclude that
\[
\|\bar{\vp}\|_1 = \|\vp\|_1 < 1,
\]
and hence $X_{\clq} : (\clm_\clq, \|\cdot\|_1) \raro \mathbb{C}$ is not a contraction. This proves the following result:

\begin{proposition}\label{prop: example of no lift}
Let $\{\alpha_1, \ldots, \alpha_n\} \subset \mathbb{R}_{\geq 0}$, suppose $\alpha_p, \alpha_q \neq 0$ for some $p\neq q$, and
\[
\sum_{i=1}^{n} \alpha_i^2 = 1.
\]
Let $z_0 \in \T^n$, and let $\vp_i$ be an inner function independent of the variable $z_i$, and suppose $\vp_i(z_0)$ is well-defined and
\[
|\vp_i(z_0)| = 1,
\]
for all $i=1, \ldots, n$. Define $\vp \in \cls(\D^n)$ by
\[
\vp = \sum_{i=1}^{n} \alpha_i \bar{\beta}_i \zeta_i \vp_i,
\]
where $\beta_i = (\zeta_i \vp_i)(z_0)$ and $\zeta_i:= \prod_{j\neq i} z_j$ for all $i=1, \ldots, n$. Then $S_\vp$ on $\clq = \ker (\prod_{i=1}^{n} T_{z_i}^*)$ does not admit any lift.
\end{proposition}

In the above, saying that $\varphi_i(z_0)$ is well-defined refers to the fact that the radial limit of $\varphi_i$ exists at $z_0 \in \T^n$.

\section{Recovering Sarason's lifting theorem}\label{sec: recover Sarason}

In this section, we explain how to recover Sarason's commutant lifting theorem from Theorem \ref{thm: new CLT 1}. We will employ several tools (just like Sarason) that are commonly used and are valid only in one variable function theory. We start with the Beurling theorem \cite{Beurling}. Let $\clq \subsetneqq H^2(\T)$ be a closed subspace. Then $\clq$ is a quotient module if and only if there exists an inner function $\theta \in H^\infty(\D)$ such that $\clq = \clq_\theta$, where
\[
\clq_\theta := H^2(\T) \ominus \theta H^2(\T).
\]
Observe that $\theta H^2(\T)$ is a closed subspace (as $T_\theta$ is an isometry on $H^2(\T)$) and
\[
\clq_\theta \cong H^2(\T)/ \theta H^2(\T).
\]
Therefore, quotient modules of $H^2(\T)$ are inner function based - a typical one variable phenomenon (see Rudin \cite{WR} for counterexamples in several variables). In the following, we prove a key result.

\begin{lemma}\label{lemma: 1 var qm equal}
Let $\theta \in H^\infty(\D)$ be an inner function. Then
\[
{\clq}^{conj}_{\theta} \oplus z H^2(\T) = \overline{\theta} (z H^2(\T)).
\]
\end{lemma}
\begin{proof}
Let $g \in \clq_\theta$. Then $\bar{g} \in {\clq}^{conj}_{\theta}$, and hence, for each $m \geq 0$, we have
\[
\begin{split}
\langle \theta \bar{g}, \bar{z}^m \rangle_{L^2(\T)} & = \overline{\langle \bar{\theta} g, z^m \rangle}_{L^2(\T)}
\\
& = \overline{\langle g, \theta z^m \rangle}_{H^2(\T)}
\\
& = 0,
\end{split}
\]
as $\theta z^m \in \clq_\theta^{\perp}$. This implies $\theta {\clq}^{conj}_{\theta} \subseteq z H^2(\T)$ and hence ${\clq}^{conj}_{\theta} \subseteq \bar{\theta} (z H^2(\T))$. Also, for all $h \in H^2(\T)$, since
\[
z h = \overline{\theta}(\theta zh) = \overline{\theta}(z \theta h),
\]
it follows that $ z H^2(\T) \subseteq \overline{\theta} (z H^2(\T))$. Therefore
\[
{\clq}^{conj}_{\theta} \oplus z H^2(\T) \subseteq \overline{\theta} (z H^2(\T)).
\]
For the reverse inclusion, first we observe that for $f \in \clq_\theta$ and $m \geq 1$, since
\[
\begin{split}
\langle \overline{\theta} z f, z^m \rangle_{L^2(\T)} & = \langle z f, z \theta z^{m-1}\rangle_{H^2(\T)}
\\
& = \langle f, \theta z^{m-1} \rangle_{H^2(\T)}
\\
& = 0,
\end{split}
\]
it follows that $\bar{\theta} z \clq_\theta \perp z H^2(\T)$, and hence $\bar{\theta} z \clq_\theta \subseteq H^2(\T)^{conj}$. On the other hand, we know
\[
{H^2(\T)}^{conj} = {\clq}^{conj}_{\theta} \oplus (\theta H^2(\T))^{conj}.
\]
In view of this, for each $f \in \clq_\theta$ and $g \in H^2(\T)$, we further compute
\[
\begin{split}
\langle \overline{\theta} z f, \overline{\theta} \bar{g} \rangle_{L^2(\T)} & = \langle z f, \bar{g}\rangle_{L^2(\T)}
\\
& = \langle f, \bar{z} \bar{g} \rangle_{L^2(\T)}
\\
& = 0,
\end{split}
\]
which implies that $\overline{\theta} z \clq_\theta \perp (\theta H^2(\T))^{conj}$. As a result, $\bar{\theta} z \clq_\theta \subseteq \clq_\theta^{conj}$. Finally
\[
zH^2(\T) = z \clq_\theta \oplus z \theta H^2(\T),
\]
yields
\[
\begin{split}
\overline{\theta} z H^2(\T) & = \overline{\theta} z \clq_\theta + z H^2(\T)
\\
& \subseteqq {\clq}^{conj}_{\theta} + z H^2(\T),
\end{split}
\]
and completes the proof of the lemma.
\end{proof}

We are now almost ready to prove Sarason's commutant lifting theorem. Just one more result is required with regard to representations of polynomials as the sum of $H^\infty(\D)$-functions. Since this result holds true in several variables and is of independent interest, we prove it in the later part of this paper (see Proposition \ref{prop: poly decomp}).

\begin{theorem}\label{thm: recov sarason}
Contractive module maps on quotient modules of $H^2(\T)$ are liftable.
\end{theorem}
\begin{proof}
Since we are dealing with one variable quotient module, we fix a quotient module $\clq_\theta$ of $H^2(\T)$ corresponding to an inner function $\theta \in H^\infty(\D)$. Since $\clm_1 = \{0\}$ and $H^2_0(\T) = z H^2(\T)$, it follows that
\[
\clm_{\clq_\theta} = {\clq}^{conj}_{\theta} \oplus z H^2(\T),
\]
and hence Lemma \ref{lemma: 1 var qm equal} yields a compact form of $\clm_{\clq_\theta}$ as
\[
\clm_{\clq_\theta} = \overline{\theta}(z H^2(\T)).
\]
Let $X \in \clb_1(\clq)$ and let $\psi = X (P_{\clq_\theta} 1)$. In view of the above and Theorem \ref{thm: new CLT 1}, it is enough to prove that $X_{\clq_\theta}: (\clm_{\clq_\theta}, \|\cdot\|_1) \rightarrow \mathbb{C}$ is a contraction, where
\[
X_{\clq_\theta} (\bar{\theta} f) = \int_{\T} \psi \bar{\theta} f \, d\mu,
\]
for all $f \in z H^2(\T)$. To this end, fix $f \in z H^2(\T)$. Then $f \in H^2(\T)$ and $f(0) = 0$. There exists a sequence of polynomials $\{p_m\}_{m \geq 0} \subseteq \mathbb{C}[z]$ such that
\[
p_m(0) = 0,
\]
for all $m\geq 0$, and
\[
p_m \longrightarrow f \text{ in } H^2(\T).
\]
Using the contractive containment $H^2(\T) \hookrightarrow H^1(\T)$, we see that
\[
p_m \rightarrow f \text{ in } H^1(\T).
\]
It also follows that
\begin{equation}\label{eqn: proof rec sar 1}
\bar{\theta} p_m  \rightarrow \bar{\theta} f,
\end{equation}
in both $L^2(\T)$ and  $L^1(\T)$. Then
\[
\int_{\T} \psi \bar{\theta} p_m  d\mu \raro \int_{\T} \psi \bar{\theta} f  d\mu ,
\]
and
\[
\|\bar{\theta} p_m\|_1 \raro \|\bar{\theta} f\|_1,
\]
and hence it is enough to prove that
\[
\Big| \int_{\T} \psi \bar{\theta} p \, d\mu\Big| \leq \|\bar{\theta} p\|_1,
\]
for all $p \in \mathbb{C}[z]$ such that $p(0) = 0$. Fix such a polynomial $p$. Consider the inner-outer factorization of $p$ as
\[
p = \eta h,
\]
where $\eta$ is an inner function, $h$ is outer, and $\eta(0) = 0$. Since $p \in H^\infty(\D)$, it is follows that $h \in H^\infty(\D)$. Using the fact that $\sqrt{h} \in H^\infty(\D) \subseteq H^2(\T)$, we rewrite $p$ as
\[
p = (\eta \sqrt{h}) \sqrt{h}.
\]
It is easy to see that
\[
\|p\|_1 = \|\sqrt{h}\|_2^2.
\]
Moreover, we have a sequence of polynomials $\{q_t\}_{t \geq 0} \subseteq \mathbb{C}[z]$ such that
\[
q_{t} \longrightarrow \sqrt{h} \text{ in } H^2(\T).
\]
As $\eta \sqrt{h} \in H^\infty(\D)$, we have
\[
\langle \psi, \theta  \overline{q_t \eta \sqrt{h}} \rangle \longrightarrow \langle \psi, \theta \overline{\sqrt{h} \eta \sqrt{h}} \rangle,
\]
and then, rewriting $\sqrt{h} \eta \sqrt{h} = p$,  we conclude that
\[
\langle \psi, \theta \overline{q_t \eta \sqrt{h}} \rangle \longrightarrow \langle \psi, \theta \bar{p} \rangle = \int_\T \psi \bar{\theta} p d\mu,
\]
as $t \raro \infty$. Since $(\eta \sqrt{h})(0) = 0$, Lemma \ref{lemma: 1 var qm equal} implies
\[
\theta \overline{\eta \sqrt{h}} \in \clq_\theta \oplus \overline{z H^2(\T)},
\]
and consequently
\[
\tilde{h}:= P_{H^2(\T)} (\theta \overline{ \eta \sqrt{h}}) \in \clq_\theta.
\]
Then, recalling $\psi = X(P_{\clq_\theta} 1)$, we compute
\[
\begin{split}
\langle \psi, \theta \overline{q_t \eta \sqrt{h}} \rangle & = \langle \psi q_t, \theta \overline{\eta \sqrt{h}} \rangle
\\
& = \langle P_{H^2(\T)} \psi q_t, P_{H^2(\T)} \theta \overline{\eta \sqrt{h}} \rangle
\\
& =  \langle P_{H^2(\T)} \psi q_t, \tilde{h} \rangle
\\
& =  \langle P_{\clq_\theta} \psi q_t,\tilde{h} \rangle.
\end{split}
\]
We also observe, for a general polynomial $r \in \mathbb{C}[z]$, that
\[
\begin{split}
X P_{\clq_\theta} r & = X r(S_z) P_{\clq_\theta} 1
\\
& = r(S_z) X P_{\clq_\theta} 1
\\
& = r(S_z) \psi,
\end{split}
\]
that is, $X P_{\clq_\theta} r = P_{\clq_\theta} r\psi$. It is important to note that (by virtue of Proposition \ref{prop: poly decomp})
\[
P_{\clq_\theta} r \in H^\infty(\D).
\]
Since $\{q_t\}_{t \geq 0} \subseteq \mathbb{C}[z]$, we conclude
\[
\langle \psi, \theta \overline{q_t \eta \sqrt{h}} \rangle = \langle X P_{\clq_\theta} q_t, \tilde{h} \rangle,
\]
and hence
\[
\Big|\langle X P_{\clq_\theta} q_t, \tilde{h} \rangle\Big| \longrightarrow \Big|\int_\T \psi \bar{\theta} p d\mu\Big|,
\]
as $t \raro \infty$. But, $\|X\| \leq 1$, and $\|\tilde{h}\| \leq \|\sqrt{h}\|$, and hence
\[
\Big|\langle X P_{\clq_\theta} q_t, \tilde{h} \rangle\Big| \leq \|q_t\|_2 \|\sqrt{h}\|_2.
\]
As $t \raro \infty$, we have (note that $\theta$ is an inner function)
\[
\|q_t\|_2 \|\sqrt{h}\|_2 \raro \|\sqrt{h}\|^2_2 = \|p\|_1 = \|\bar{\theta} p\|_1,
\]
 and hence
\[
\Big|\int_\T \psi \bar{\theta} p d\mu\Big| \leq \|\bar{\theta} p\|_1,
\]
which completes the proof of the theorem.
\end{proof}

Sarason's proof of the above theorem used similar one-variable tools. Moreover, we point out that Sarason’s theorem proves more than we have recovered above. In fact, in his setting, the norm of the given commutator on the model space is equal to the norm of the Schur function (see the discussion preceding the first commutative diagram in the introductory section).

\section{Other results}\label{sec: con remarks}

In this section, we present a variety of results with varying flavors. First, we present a solution to the Carath\'{e}odory-Fej\'{e}r interpolation problem on $\D^n$. Then we discusses the interpolation problem from the standpoint of Pick matrix positivity. The lifting theorem for the Bergman space over $\D^n$ is then compared, followed by decompositions of polynomials in light of Beurling-type quotient modules of $H^2(\T^n)$.

\subsection{Carath\'{e}odory-Fej\'{e}r interpolation}\label{subsect: Caratheodory}
We use the notations that were introduced in Section \ref{sec: examples of hom qm}. Recall that for $t \in \Z_+$, $H_t \subseteq \mathbb{C}[z_1, \ldots, z_n]$ is the complex vector space of homogeneous polynomials of degree $t$. Moreover, for each $m \in \mathbb{N}$, define the finite-dimensional homogeneous quotient module $\clq_m$ of $H^2(\T^n)$ by
\[
\clq_m := \bigoplus_{t=0}^m H_t.
\]
Fix a natural number $m$. Given $p \in \mathbb{C}[z_1, \ldots, z_n]$, it follows that $p \in \clq_m$ if and only if $\text{deg} p \leq m$. In the context of $\cls(\D^n)$, the Carath\'{e}odory-Fej\'{e}r interpolation problem asks the following: Given a polynomial $p \in \clq_m$, when does there exist a function $f \in \clq_m^\perp$ such that
\[
p \oplus f \in \cls(\D^n)?
\]
Here and in what follows, $p \oplus f$ is in the sense of the direct sum $\clq_m \oplus \clq_m^\perp$. This formulation of the Carath\'{e}odory-Fej\'{e}r interpolation problem is more appropriate for the case of $n>1$, see \cite[page 670]{Barik et al}.

The following is an interpretation of the Carath\'{e}odory-Fej\'{e}r problem in terms of commutant lifting.

\begin{proposition}\label{prop: caratheodory}
Let $p \in \clq_m$. There exists $f \in \clq_m^\perp$ such that $p \oplus f \in \cls(\D^n)$ if and only if $S_p$ is a contraction and admits a lift.
\end{proposition}
\begin{proof}
Suppose there exists a function $f \in \clq_m^\perp$ such that
\[
\vp:=p \oplus f \in \cls(\D^n).
\]
For each $q \in \clq_m$, we have
\[
\begin{split}
S_\vp q & = P_{\clq_m} T_\vp q
\\
& = P_{\clq_m} (p \oplus f) q
\\
& = P_{\clq_m} (pq) + P_{\clq_m}(f q).
\end{split}
\]
But, $\clq_m^\perp$ is a submodule and $q$ is a polynomial. This implies $fq \in \clq_m^\perp$, and consequently
\[
S_\vp q = P_{\clq_m} (pq).
\]
On the other hand, $q \in \clq_m$ and
\[
p \in \clq_m \subseteq \mathbb{C}[z_1, \ldots, z_n]
\]
yield
\[
\begin{split}
P_{\clq_m} (pq) & = P_{\clq_m} T_p|_{\clq_m} q
\\
& = S_p q,
\end{split}
\]
which proves that $S_\vp = S_p$. The contractivity of $S_p$ also follows from the same of $S_\vp$ (recall that $\|\vp\|_\infty \leq 1$).

\noindent For the reverse direction, suppose $S_p \in \clb_1(\clq_m)$ admits a lift. Then there exists $\vp \in \cls(\D^n)$ such that $S_p = S_\vp$. Using $1 \in \clq_m$, it follows that
\[
\begin{split}
p & = S_p 1
\\
& = S_\vp 1
\\
& = P_{\clq_m} \vp,
\end{split}
\]
and hence there exists $f \in \clq_m^\perp$ such that $\vp = p \oplus f$. This completes the proof of the proposition.
\end{proof}

We are now ready for the solution to the Carath\'{e}odory-Fej\'{e}r interpolation problem. We will apply our commutant lifting theorem to the above. In view of Theorem \ref{thm: new CLT 1}, we set
\[
\clm_{\clq_m} = \clq^{conj}_m \dotplus (\clm_n \dotplus H^2_0(\T^n)).
\]
Recall that
\[
\clm_n = L^2(\T^n) \ominus ({H^2(\T^n)}^{conj} \dotplus H^2(\T^n)).
\]

\begin{corollary}\label{cor: Caratheodory}
Given $p \in \clq_m$, there exists $f \in \clq_m^\perp$ such that $p \oplus f \in \cls(\D^n)$ if and only if $\mathfrak{C}_{\clz, \clw}: (\clm_{\clq_m}, \|\cdot\|_1) \raro \mathbb{C}$ is a contraction, where
\[
\mathfrak{C}_{\clz, \clw} (g) = \int_{\T^n} p g \,d\mu \qquad (g \in \clm_{\clq_m}).
\]
\end{corollary}
\begin{proof}
By Theorem \ref{thm: new CLT 1} and the preceding proposition, the assertion is equivalent to the contractivity of the functional $\chi_{{\clq_m}}$ on $(\clm_{\clq_m}, \|\cdot\|_1)$, where
\[
\chi_{{\clq_m}} g = \int_{\T^n} \psi g \,d\mu \qquad (g \in \clm_{\clq_m}),
\]
and $\psi = S_p (P_{\clq_m} 1)$. However, $1 \in \clq_m$ implies $P_{\clq_m} (1) = 1$, and $p \in \clq_m$ implies $S_p (1) = p$. Then
\[
\chi_{{\clq_m}} = \mathfrak{C}_{\clz, \clw} \text{ on } \clm_{\clq_m},
\]
completes the proof of the corollary.
\end{proof}

We refer the reader to Eschmeier, Patton and Putinar \cite{Patton}, and Woerdeman \cite{Hugo} for the Carath\'{e}odory interpolation problem in the context of Agler-Herglotz class functions and Agler-Herglotz-Nevanlinna formula on the polydisc. Also see the paper by Kalyuzhnyi-Verbovetzkii \cite{Dimitry}.

\subsection{Weak interpolation}\label{subsect: weak interpolation} Given $\clz = \{z_i\}_{i=1}^m \subset \D^n$ and $\clw = \{w_i\}_{i=1}^m \subset \D$, we define the $m \times m$ \textit{Pick matrix} $\mathfrak{P}_{\clz, \clw}$ as
\[
\mathfrak{P}_{\clz, \clw} = \Big((1 - w_i \bar{w}_j) \mathbb{S}(z_i, z_j) \Big)_{i,j=1}^m.
\]
Recall that a matrix $(a_{ij})_{m\times m}$ is positive semi-definite (in short $(a_{ij})_{m\times m} \geq 0$) if
\[
\sum_{i,j=1}^{m} \bar{\alpha}_i \alpha_j a_{ij} \geq 0,
\]
for all scalars $\{\alpha_i\}_{i=1}^m \subseteq \mathbb{C}$.

\begin{definition}
A set of distinct points $\clz = \{z_i\}_{i=1}^m \subset \D^n$ is said to be a Pick set if, for $\clw = \{w_i\}_{i=1}^m \subset \D$ satisfying
\[
\mathfrak{P}_{\clz, \clw} \geq 0,
\]
there exists $\vp \in \cls(\D^n)$ such that $\vp(z_i) = w_i$ for all $i=1, \ldots, m$.
\end{definition}

This definition is in view of the classical Pick positivity and the Nevanlinna-Pick interpolation on $\D$. We need another definition along the lines of Sarason's commutant lifting theorem:

\begin{definition}
A quotient module $\clq \subseteq H^2(\T^n)$ satisfies the commutant lifting property if every contraction on $\clq$ admits lifting.
\end{definition}

In other words, for a module map $X \in \clb_1(\clq)$, there exists $\vp \in \cls(\D^n)$ such that $X = S_\vp$. Now we use Sarason's trick to prove the Nevanlinna-Pick interpolation but in the setting of $\cls(\D^n)$ for any $n \geq 1$. The proof is standard and follows in Sarason's footsteps.

\begin{proposition}\label{prop: Pick set and CLT}
Let $\mathcal{Z} = \{z_j\}_{j=1}^m \subset \D^n$ be a set of $m$ distinct points. Then $\mathcal{Z}$ is a Pick set if and only if $\clq_{\mathcal{Z}}$ satisfies the commutant lifting property, where
\[
\clq_\clz = \text{span} \{\mathbb{S}(\cdot, z_i): i=1, \ldots, m\}.
\]
\end{proposition}
\begin{proof}
We begin with a simple observation. Given $\clw = \{w_i\}_{i=1}^m \subset \D$, we define $X \in \clb(\clq_\clz)$ by (note that $\clq_\clz$ is a finite-dimensional Hilbert space)
\[
X \mathbb{S}(\cdot, z_i) = \bar{w}_i \mathbb{S}(\cdot, z_i) \qquad (i=1, \ldots, m).
\]
By Lemma \ref{eqn: Y star S}, it follows that $X^*$ is a module map. Moreover, we have
\[
\Big \langle (I_{\clq_\clz} - X^* X) \Big(\sum_{j=1}^{m} \alpha_j \mathbb{S} (\cdot, z_j)\Big), \Big(\sum_{i=1}^{m} \alpha_i \mathbb{S} (\cdot, z_i)\Big) \Big \rangle = \sum_{i,j=1}^{m} \alpha_j \bar{\alpha}_i (1 - w_i \bar{w}_j) \mathbb{S}(z_i, z_j),
\]
for all scalars $\{\alpha_i\}_{i=1}^m \subset \mathbb{C}$. It follows that $X$ is a contraction if and only if
\[
\mathfrak{P}_{\clz, \clw} \geq 0.
\]
Now suppose that $\clz$ is a Pick set, and let $Y \in \clb_1(\clq_\clz)$ be a module map. We claim that $Y$ has a lift. If we define $X:= Y^*$, then we are precisely in the setting of the above discussion. The contractivity of $X$ (as $\|Y^*\| \leq 1$) then implies that the Pick matrix is positive, that is, $\mathfrak{P}_{\clz, \clw} \geq 0$. There exists $\vp \in \cls(\D^n)$ such that $\vp(z_i) = w_i$ for all $i=1, \ldots, m$. Then
\[
Y^* \mathbb{S}(\cdot, z_j) = T_\vp^* \mathbb{S}(\cdot, z_j) \qquad (j=1,\ldots, m),
\]
and we conclude that $Y^* = T_{\vp}^*|_{\clq_\clz}$, or equivalently, $Y = S_\vp$.

\noindent To show the converse, assume that $\clq_\clz$ satisfies the commutant lifting property. Let $\clw = \{w_i\}_{i=1}^m \subset \D$, and suppose $\mathfrak{P}_{\clz, \clw} \geq 0$. Then $X$, as defined at the beginning of the proof, is a contraction, and hence $X = S_\vp$ for some $\vp \in \cls(\D^n)$. It is now routine to check that $\vp(z_i) = w_i$ for all $i=1, \ldots, m$.
\end{proof}

In the case of $n=1$, the classical Nevanlinna Pick interpolation theorem now follows directly from Sarason's lifting theorem. The above formulation also works verbatim the same way as for multiplier spaces for general reproducing kernel Hilbert spaces over domains in $\mathbb{C}^n$ (including the open unit ball in $\mathbb{C}^n$).

In view of the above proposition, we conclude that the solution to the interpolation problem in terms of Pick positivity is simply equivalent to the commutant lifting problem for quotient modules of the form $\clq_\clz$ for finite subsets $\clz \subseteq \D^n$. Again, this is true for general multiplier spaces.

\subsection{Bergman space and lifting}\label{subsect: Bergman CLT} Although all of the observations in this subsection hold true for weighted Bergman spaces (even for a large class of reproducing kernel Hilbert spaces) over $\D^n$ along with verbatim proofs, we will stick to the Bergman space only. Denote by $A^2(\D^n)$ the Bergman space over $\D^n$. Recall that an analytic function $f$ on $\D^n$ is in $A^2(\D^n)$ if and only if
\[
\|f\|_{A^2(\D^n)}:= \Big(\int_{\D^n} |f(z)|^2 \, d\sigma(z)\Big)^{\frac{1}{2}} < \infty,
\]
where $d\sigma(z)$ denotes the normalized volume measure on $\D^n$. We know that $A^2(\D^n)$ is a reproducing kernel Hilbert space corresponding to the Bergman kernel
\[
K(z,w) = \prod_{i=1}^n \frac{1}{(1 - z_i \bar{w}_i)^2} \qquad (z, w \in \D^n).
\]
Recall that the multiplier space of $A^2(\D^n)$ is again $H^\infty(\D^n)$, which for simplicity of notation (or, to avoid confusion), we denote by $\clm(A^2(\D^n))$. In other words
\[
\clm(A^2(\D^n)) = H^\infty(\D^n).
\]
For each $\vp \in \clm(A^2(\D^n))$, the map $f \in A^2(\D^n) \mapsto \vp f \in A^2(\D^n)$ defines a multiplication operator on $A^2(\D^n)$, which we denote by $M_\vp$.

Let $\clq \subseteq A^2(\D^n)$ be a quotient module (that is, $\clq$ is closed and $M^*_{z_i} \clq \subseteq \clq$ for all $i=1, \ldots, n$). For each $\vp \in H^\infty(\D^n)$, set
\[
B_{\vp} = P_{\clq} M_\vp|_{\clq}.
\]
Let $X \in \clb_1(\clq)$ be a module map, that is, $X B_{z_i} = B_{z_i} X$ for all $i=1, \ldots, n$. We say that $X$ is \textit{liftable} or $X$ has a \textit{lift} if there exists $\vp \in H^\infty(\D^n) = \clm(A^2(\D^n))$ such that
\[
X = B_\vp,
\]
and
\begin{equation}\label{eqn: Berg mult norm}
\|B_{\vp}\|_{\clb(A^2(\D^n))} \leq 1.
\end{equation}
We are interested in the commutant lifting for finite-dimensional zero-based quotient modules of $A^2(\D^n)$. For a set of distinct points $\clz = \{z_i\}_{i=1}^m \subset \D^n$, we define (following Section \ref{sec: weak lift}) the $m$-dimensional zero-based quotient module $\clb_\clz \subseteq A^2(\D^n)$ as
\[
\clb_\clz = \text{span} \{K(\cdot, z_i): i=1, \ldots, m\} \subseteq A^2(\D^n).
\]
At the same time, keep in mind that $\clq_\clz$ is also a zero-based quotient module of $H^2(\T^n)$ (again, see the preceding subsection or Section \ref{sec: weak lift}), where
\[
\clq_\clz = \text{span} \{\mathbb{S}(\cdot, z_i): i=1, \ldots, m\} \subseteq H^2(\T^n).
\]
Note that module maps on $\clq_\clz$ are parameterized by $m$ scalars. To be more precise, let $X \in \clb(\clq_\clz)$. Then $X$ is a module map if and only if there exists $\{w_i\}_{i=1}^m \subset \mathbb{C}$ such that
\[
X^* \mathbb{S}(\cdot, z_i) = w_i \mathbb{S}(\cdot, z_i),
\]
for all $i=1, \ldots, m$. This was observed in Lemma \ref{eqn: Y star S}. The same conclusion and proof apply to $\clb_\clz$. Therefore, a module map $X \in \clb(\clq_\clz)$ is associated with $\{w_i\}_{i=1}^m \subseteq \mathbb{C}$, which further defines a module map $\tilde X \in \clb(\clb_\clz)$ as
\[
\tilde{X}^* K(\cdot, z_i) = w_i K(\cdot, z_i),
\]
for all $i=1, \ldots, m$. Consequently, we have the bijective correspondence
\[
X \in \clb(\clq_\clz) \longleftrightarrow \tilde{X} \in \clb(\clb_\clz).
\]

In the case of $n=1$, the problem of commutant lifting for quotient module $\clb_\clz$ of $A^2(\D)$ was studied in the thesis of Sultanic \cite{Sad -1}. While she was focused on finite-dimensional quotient modules of $A^2(\D)$, but the zero-based quotient modules played the most crucial role. Here we aim at proving the following proposition:

\begin{proposition}
Let $\clz = \{z_i\}_{i=1}^m \subset \D^n$ be a set of distinct points, and let $X \in \clb_1(\clq_\clz)$ be a module map. Then $X$ on $\clq_\clz$ is liftable if and only if $\tilde X$ on $\clb_\clz$ is liftable.
\end{proposition}
\begin{proof}
We start by stating a general (and well known) fact: Let $\vp \in \clm(A^2(\D^n))$. Then the operator norm (or multiplier norm) of $M_\vp$ on $A^2(\D^n)$ is given by
\[
\|M_\vp\|_{\clb(A^2(\D^n))} = \|\vp\|_\infty.
\]
Indeed, for $f \in A^2(\D^n)$, we have
\[
\begin{split}
\|\vp f\|_{A^2(\D^n)} & = \Big( \int_{\D^n} |\vp f|^2 d\sigma\Big)^{\frac{1}{2}}
\\
& \leq \Big( \int_{\D^n} \|\vp\|_\infty^2 |f|^2 d\sigma\Big)^{\frac{1}{2}}
\\
& \leq \|\vp\|_\infty \Big( \int_{\D^n} |f|^2 d\sigma\Big)^{\frac{1}{2}},
\end{split}
\]
that is, $\|M_\vp\|_{\clb(A^2(\D^n))} \leq \|\vp\|_\infty$. On the other hand, for each $w \in \D^n$,
\[
\begin{split}
\vp(w) & = \frac{1}{\|K(\cdot, w)\|^2} \langle K(\cdot, w), \overline{\vp(w)} K(\cdot, w) \rangle
\\
& = \frac{1}{\|K(\cdot, w)\|^2} \langle K(\cdot, w), T_{\vp}^* K(\cdot, w) \rangle
\\
& =  \Big\langle T_{\vp} \Big(\frac{K(\cdot, w)}{\|K(\cdot, w)\|}\Big), \frac{K(\cdot, w)}{\|K(\cdot, w)\|} \Big\rangle,
\end{split}
\]
implies that $|\vp(w)| \leq \|M_\vp\|_{\clb(A^2(\D^n))}$, and completes the proof of the claim. Now, suppose that $\tilde{X}$ on $\clb_\clz$ is liftable, that is $\tilde{X} = B_\vp$ for some $\vp \in \clm(A^2(\D^n)) = H^\infty(\D^n)$ with $\|M_\vp\|_{\clb(A^2(\D^n))} \leq 1$. In view of the above observation, we have $\vp \in \cls(\D^n)$. Suppose $\{w_i\}_{i=1}^m \subset \mathbb{C}$ be the scalars corresponding to $\tilde{X}$, that is
\[
\tilde{X}^* K(\cdot, z_i) = w_i K(\cdot, z_i),
\]
for all $i=1, \ldots, m$. This and the equality $\tilde{X} = B_\vp$ imply that
\[
\vp(z_i) = \bar{w}_i \qquad (i=1, \ldots, m),
\]
and hence $X^* = S_\vp^*$. Therefore, $X = S_\vp$, and hence $\vp$ is a lift of $X$. Proof of the reverse direction is similar.
\end{proof}

In other words, the lifting problem on zero-based quotient modules of $A^2(\D^n)$ is equivalent to the lifting problem on zero-based quotient modules of $H^2(\T^n)$. In the case $n=1$, for a module map $\tilde{X} \in \clb_1(\clb_\clz)$, if $\|X\|_{\clb(\clq_\clz)} \leq 1$, then $\tilde{X}$ can be lifted (thanks to Sarason). On the other hand, if ${X} \in \clb_1(\clq_\clz)$ is a module map, then automatically $\tilde{X} \in \clb_1(\clb_\clz)$, and hence $X$ has a lift.

\subsection{Decompositions of polynomials}\label{subsect: polynomial} In this subsection, we decompose polynomials with respect to Beurling-type quotient modules of $H^2(\T^n)$. This result has already been used ($n=1$ case) to recover Sarason's commutant lifting theorem (see Theorem \ref{thm: recov sarason}).

A quotient module $\clq \subseteq H^2(\T^n)$ is said to be of \textit{Beurling type} if there exists an inner function $\vp \in \cls(\D^n)$ such that
\[
\clq = (\vp H^2(\T^n))^\perp.
\]
Recall that all one variable quotient modules are of Beurling type \cite{Beurling}.

\begin{proposition}\label{prop: poly decomp}
Let $\vp \in \cls(\D^n)$ be an inner function, and let $p \in \mathbb{C}[z_1, \ldots, z_n]$. Write
\[
p = f \oplus g \in \vp H^2(\T^n) \oplus (\vp H^2(\T^n))^\perp.
\]
Then $f, g \in H^\infty(\D^n)$.
\end{proposition}
\begin{proof}
It is enough to prove that $f \in H^\infty(\D^n)$. It is also enough to consider $p$ as a monomial. Fix $\bk \in \Z_+^n$, and suppose
\[
z^{\bk} = f \oplus g \in \vp H^2(\T^n) \oplus (\vp H^2(\T^n))^\perp.
\]
Let $\bl \in \Z_+^n$, and suppose $l_i > k_i$ for some $i =1, \ldots, n$. Since $T_z^{*\bl} (z^{\bk}) = 0$, it follows that
\[
T_z^{*\bl} f = - T_z^{*\bl} g.
\]
Since $g$ is in the quotient module $(\vp H^2(\T^n))^\perp$, we conclude that
\[
T_z^{*\bl} f \in (\vp H^2(\T^n))^\perp.
\]
Now there exists $f_1 \in H^2(\T^n)$ such that $f = \vp f_1$. Consequently
\[
T_z^{*\bl} f = T_z^{*\bl} \vp f_1 \in (\vp H^2(\T^n))^\perp,
\]
and hence
\[
\langle T_z^{*\bl} \vp f_1, \vp h \rangle = 0,
\]
for all $h \in H^2(\T^n)$. Then, $T_\vp^* T_z^{*\bl} = T_z^{*\bl} T_\vp^*$ and $T_\vp^* T_\vp = I$ yield
\[
\begin{split}
\langle T_z^{*\bl} f_1, h \rangle & = \langle T_z^{*\bl} \vp f_1, \vp h \rangle
\\
& = 0,
\end{split}
\]
for all $h \in H^2(\T^n)$ and $l \in \Z_+^n$ such that $l_i > k_i$ for some $i =1, \ldots, n$. Therefore
\[
f_1 \in \bigcap_{|l|=|k|+1} \ker T_z^{*l}.
\]
and hence
\[
f_1 \in \text{span} \{z^t: t \in \Z_+^n, |t| \leq |k|+1\}.
\]
We conclude that
\[
f \in \text{span} \{z^t \vp: t \in \Z_+^n, |t| \leq |k|+1\} \subseteq H^\infty(\D^n).
\]
This completes the proof of the proposition.
\end{proof}

A similar question could be posed for other classes of functions. What about the decomposition of a rational function with respect to a Beurling decomposition, for example?

\section{Concluding remarks}\label{sec: final}

We start off by commenting on the commutant lifting theorem. Let us recall Ball, Li, Timotin, and Trent's commutant lifting theorem \cite[Theorem 5.1]{Ball et}, which is only relevant for $n=2$ case in our context.

\begin{theorem}
Let $\clq \subseteq H^2(\T^2)$ be a quotient module, and let $X \in \clb_1(\clq)$ be a module map. Then $X$ admits a lift if and only if there exist positive operators $G_1, G_2 \in \clb(\clq)$ such that $G_1 - S_{z_2} G_1 S_{z_2}^* \geq 0$ and $G_2 - S_{z_1} G_2 S_{z_1}^* \geq 0$, and
\[
I - X X^* = G_1 + G_2.
\]
\end{theorem}

The proof is based on Agler's transfer function realization formula for functions in $\cls(\D^2)$ (which we will comment on more about below). In contrast to the preceding theorem, however, our commutant lifting theorem appears to be more explicit. For instance, Theorem \ref{thm: new CLT 1} has been validated for the examples constructed in Corollary \ref{cor: hom poly fail} (see Section \ref{subsect: example verification}).

Now we turn to the interpolation problem. We already mentioned in Section \ref{sec: 1} that the traditional approach to solving the interpolation problem in terms of the positivity of the Pick matrix (or family of Pick matrices) in higher variables produces only limited results. There is, however, likely to be one notable exception: interpolation on $\D^2$, which Agler \cite{Agler 1, Agler} pioneered in his seminal papers in the late '80s (also see \cite[Theorem 1.3]{AM Crelle}):

\begin{theorem}\label{thm: Agler inter}
Let $\{(\alpha_i, \beta_i)\}_{i=1}^m$ be a set of distinct points in $\D^2$ and let $\{w_i\}_{i=1}^m \subset \D$. There exists $\vp \in \cls(\D^2)$ such that
\[
\vp(\alpha_i, \beta_i) = w_i,
\]
for all $i=1, \ldots, m$, if and only if there exist positive semi-definite $m \times m$ matrices $\Gamma = (\Gamma_{ij})$ and $\Delta = (\Delta_{ij})$ such that
\[
(1 - \bar{w}_i w_j) = (1 - \bar{\alpha}_i \alpha_j) \Gamma_{ij} + (1 - \bar{\beta}_i \beta_j) \Delta_{ij},
\]
for all $i,j = 1,\ldots, m$.
\end{theorem}

This is clearly an analogue of the solution to the classical Nevanlinna–Pick interpolation problem (also see Cole and Wermer \cite{CW-1, CW-2, CW-3}). In a slightly different context, see Kosi\'{n}ski \cite{Kosinski} for three-point interpolation problem (also, see Cotlar and Sadosky \cite{Cora, CS 1, CS 2}). Whereas the above result appears to be abstract (particularly the existence of positive semi-definite matrices), the approach is useful in a variety of other problems. Indeed, based on the Ando dilation and the von Neumann inequality for pairs of commuting contractions \cite{Ando}, Agler derived a realization formula for Schur functions in terms of colligation matrices, which leads to the above solution to the interpolation problem. His realization formula has proven very useful in operator theory and function theory on $\D^n$, $n \geq 2$. Whereas we believe Theorem \ref{thm: interpolation} is more concrete and provides a new perspective on the interpolation problem in general, we are unsure how to relate it to Theorem \ref{thm: Agler inter}. We are also unclear about using Theorem \ref{thm: Agler inter} to validate the examples of interpolation in Theorem \ref{thm: eg interpolation} for the specific case of $n=2$.

Finally, we remark that, unlike the present case of scalar functions, the earlier lifting theorem and the solutions to the interpolation problem work equally well for the operator or vector-valued functions \cite{AM Crelle, Ball et, Ball-Trent}. The powerful $n$-variables von Neumann inequality (which is automatic in the case of $n=2$ but not so when $n>2$), like the Sz.-Nagy and Foia\c{s} \cite{NF 68} effective dilation theoretic approach appears to be a key factor. However, as previously stated, we followed a function theoretic route pioneered by Sarason in his work \cite{Sarason}. The results reported here, we think, will be also helpful in building related theories like isometric dilations for commuting contractions, several variables von Neumann inequality, Nehari problem on $\D^n$, etc., similar to Sarason's classic result.

\vspace{0.2in}

\noindent\textbf{Acknowledgement:}
We thank the referees for their valuable comments. We are also thankful to Professor Gadadhar Misra and Professor Nikolai Nikolski for their insightful comments and observations. The research of the second named author is supported in part by the Core Research Grant (CRG/2019/000908), and TARE (TAR/2022/000063) by SERB, Department of Science \& Technology (DST), Government of India.

\end{document}